\documentclass{amsart}
\usepackage[margin=1.3in]{geometry}
\usepackage{color}
\usepackage{float}
\usepackage{amsmath}
\usepackage{amssymb}
\usepackage{amsthm}
\usepackage{mathrsfs}
\usepackage{enumerate}
 
\numberwithin{equation}{section}

\newtheorem{thm}{Theorem}[section]
\newtheorem*{thm*}{Theorem}
\newtheorem{prop}[thm]{Proposition}
\newtheorem{lemma}[thm]{Lemma}

\newtheorem{cor}[thm]{Corollary}

\theoremstyle{remark}
\newtheorem{remark}[thm]{Remark}

\newcommand{\F}{\mathbb{F}}

\newcommand{\R}{\mathbb{R}}

\newcommand{\Z}{\mathbb{Z}}

\newcommand{\ep}{\varepsilon}

\newcommand{\con}{\equiv}

\newcommand{\ndiv}{\nmid}
\newcommand{\modd}[1]{\; ( \text{mod} \; #1)}
\newcommand{\bstack}[2]{#1 \atop #2}

\newcommand{\maps}{\rightarrow}
\newcommand{\intersect}{\cap}

\newcommand{\Union}{\bigcup}

\newcommand{\supp}{{\rm supp \;}}

\newcommand{\al}{\alpha}
\newcommand{\be}{\beta}

\newcommand{\del}{\delta}
\newcommand{\Del}{\Delta}

\newcommand{\sig}{\sigma}
\newcommand{\lam}{\lambda}

\newcommand{\Acal}{\mathcal{A}}
\newcommand{\Pcal}{\mathcal{P}}

\newcommand{\Gcal}{\mathcal{G}}

\newcommand{\Scal}{\mathcal{S}}

\newcommand{\Mbf}{\mathbf{M}}

\newcommand{\Sbf}{\mathbf{S}}

\newcommand{\Ebf}{\mathbf{E}}

\newcommand{\onebf}{\boldsymbol1}

\newcommand{\beq}{\begin{equation}}
\newcommand{\eeq}{\end{equation}}

\makeatletter
\def\@tocline#1#2#3#4#5#6#7{\relax
  \ifnum #1>\c@tocdepth 
  \else
    \par \addpenalty\@secpenalty\addvspace{#2}%
    \begingroup \hyphenpenalty\@M
    \@ifempty{#4}{%
      \@tempdima\csname r@tocindent\number#1\endcsname\relax
    }{%
      \@tempdima#4\relax
    }%
    \parindent\z@ \leftskip#3\relax \advance\leftskip\@tempdima\relax
    \rightskip\@pnumwidth plus4em \parfillskip-\@pnumwidth
    #5\leavevmode\hskip-\@tempdima
      \ifcase #1
       \or\or \hskip 1em \or \hskip 2em \else \hskip 3em \fi%
      #6\nobreak\relax
    \hfill\hbox to\@pnumwidth{\@tocpagenum{#7}}\par
    \nobreak
    \endgroup
  \fi}
\makeatother

\numberwithin{equation}{section}

\begin{document}

\title[Counterexamples for high-degree generalizations of Schr\"odinger]{Counterexamples for high-degree generalizations of the \\Schr\"odinger maximal operator}
\author[An]{Chen An}
\address{Duke University, 120 Science Drive, Durham NC 27708}
\email{chen.an@duke.edu}

\author[Chu]{Rena Chu}
\address{Duke University, 120 Science Drive, Durham NC 27708}
\email{rena.chu@duke.edu}

\author[Pierce]{Lillian B. Pierce}
\address{Duke University, 120 Science Drive, Durham NC 27708}
\email{pierce@math.duke.edu}

\makeatletter
\@namedef{subjclassname@2020}{%
  \textup{2020} Mathematics Subject Classification}
\makeatother

\begin{abstract}
In 1980 Carleson posed a question on the minimal regularity of an initial data function in a Sobolev space $H^s(\R^n)$ that implies pointwise convergence for the solution of the linear Schr\"odinger equation. After progress by many authors, this was recently resolved (up to the endpoint) by  Bourgain, whose counterexample construction for the Schr\"odinger maximal operator proved a necessary condition on the regularity, and  Du and Zhang, who proved a sufficient condition. 
Analogues of Carleson's question remain open for many other dispersive PDE's. We develop a flexible new method to approach such problems, and prove that for any integer $k\geq 2$, if a degree $k$ generalization of the Schr\"odinger maximal operator is bounded from $H^s(\R^n)$ to $L^1(B_n(0,1))$, then $s \geq  \frac{1}{4} + \frac{n-1}{4((k-1)n+1)}.$
In dimensions $n \geq 2$, for every degree $k \geq 3$, this is the first result  that exceeds a long-standing barrier at $1/4$. Our methods are number-theoretic, and in particular apply the Weil bound, a consequence of the truth of the Riemann Hypothesis over finite fields.
\end{abstract}

\maketitle
\section{Introduction}

Given a real-valued polynomial $P(\xi):\R^n \maps \R$, define
\[
T_t^{(P)}f(x):=\frac{1}{(2\pi)^n} \int_{\R^n} \hat{f}(\xi)e^{i(\xi \cdot x+P(\xi)t)} d\xi,
\]
acting initially on functions $f$ of Schwartz class on $\R^n$.
In the case that $P(\xi) = P_2(\xi) := |\xi|^2,$
$T_t^{(P_2)}f$ provides the solution to the linear Schr\"odinger equation, 
\[ \begin{cases}
    i \partial_t u - \Delta u=0, \quad (x,t) \in \R^n \times \R, \\
    u(x,0) = f(x), \quad x \in \R^n.
    \end{cases}
    \]
Thus the study of $T_t^{(P_2)}f$ relates to Carleson's well-known question of what degree of regularity of $f$ is required for the pointwise convergence result 
\begin{equation}\label{ptwise}
\lim_{t \to 0}T_t^{(P_2)}f(x)=f(x), \qquad \text{a.e. $x \in \R^n$.}
\end{equation}
 Precisely: what is the smallest value of $s$ for which  this holds for all $f \in H^s(\R^n)$?

When the story started in 1980, it was soon proved that $s \geq 1/4$ is necessary and $s \geq n/4$ is sufficient, so that for dimensions $n \geq 2$ the benchmarks were initially quite far apart  (see Carleson \cite[Eqn (14) p. 24]{Car80} and Dahlberg and Kenig \cite{DahKen82}). For many years, it was widely conjectured that the minimal value for which (\ref{ptwise}) holds should be $s=1/4$ in all dimensions.
In 2013, Bourgain expressed surprise that he was able to show otherwise: as he wrote in \cite{Bou13}, ``perhaps the most interesting point in this note is a disproof of what one
seemed to believe, namely that $f \in H^s(\R^n)$, $s>1/4$   should be the correct condition in arbitrary
dimension $n$.'' 
Ultimately, the resolution of Carleson's   question (up to the endpoint) arrived when  Bourgain \cite{Bou16} showed   that $s \geq s_2^*(n) := n/(2(n+1))$ is necessary, and  Du and Zhang \cite{DuZha19} showed that $s> s_2^*(n)$ is sufficient (see \S \ref{sec_lit} for further literature).

Analogous questions naturally arise for many other dispersive PDE's. These questions have been developed since the 1980's in the large literature on local smoothing and associated maximal operator estimates (which we review in \S \ref{sec_lit}), and were also raised explicitly by Bourgain \cite[\S 5]{Bou13} and Demeter and Guo \cite{DemGuo16}.
In this paper, we develop  flexible new number-theoretic strategies to construct counterexamples for generalizations of the Schr\"odinger maximal operator, with corresponding implications for convergence questions analogous to (\ref{ptwise}).
Notably,
we push  the necessary condition on $s$ above a long-standing barrier at $1/4$, analogous to the barrier Bourgain remarked upon for the  Schr\"odinger case. 

We introduce the maximal operator that lies at the heart of the matter. Given a real-valued polynomial $P$, define the maximal operator
\beq\label{max}
f \mapsto \sup_{0<t<1} |T_t^{(P)} f(x)|.
\eeq
If this maximal operator is bounded from $H^s(\R^n)$ to $L^2_{\mathrm{loc}}(\R^n)$ then for all $f \in H^s(\R^n)$
the pointwise convergence property holds:
\beq\label{ptwiseP}
\lim_{t\maps 0}T_t^{(P)}f(x) = f(x) \qquad \text{for a.e. $x \in \R^n$.}
\eeq
On the other hand, if this maximal operator 
is unbounded as an operator from $H^s(\R^n)$ to $L^1_{\mathrm{loc}}(\R^n)$, then (\ref{ptwiseP})
must fail for some $f \in H^s(\R^n)$, by the Stein-Nikishin maximal principle. (See e.g. \cite[Appendix A]{Pie20} for standard arguments to deduce these relationships.)
Thus  Bourgain's definitive result that $s \geq s_2^*(n)$ is necessary  for   (\ref{ptwise}) followed from showing that for each $s< s_2^*(n)$ the maximal operator (\ref{max}) with $P(\xi)=P_2(\xi) = |\xi|^2$ is unbounded from $H^s(\R^n)$ to $L^1(B_n(0,1))$, where $B_n(0,1)$ denotes the unit ball centered at the origin in $\R^n$.

In this paper, we study higher-degree analogues of the Schr\"odinger maximal operator, and prove the first necessary condition on $s$ that goes beyond $1/4$, both for the maximal estimate and for the convergence property. We focus on the family of polynomials defined for any integer $k \geq 2$ by
\[ P_k(\xi) := \xi_1^k + \cdots + \xi_n^k,\]
with associated maximal operator
\[f \mapsto \sup_{0<t<1} |T_t^{(P_k)} f|.\]
For dimension $n=1$, for every $k \geq 2,$ this maximal operator is bounded from $H^s(\R)$ to $L^1_{\mathrm{loc}}(\R)$ if and only if $s \geq 1/4$, and   the convergence property (\ref{ptwiseP}) holds if and only if $s \geq 1/4$; see \cite{DahKen82,KenRui83,Sjo87,Veg88,KPV91,Sjo97,Sjo98}.
For dimensions $n \geq 2$, degree $k=2$ is the Schr\"odinger case   resolved by Bourgain \cite{Bou16} and Du-Zhang \cite{DuZha19}. For degrees $k \geq 3,$ previous literature left a gap: the behavior for $1/4 \leq s \leq 1/2$ was unknown; see \cite{BenDev91,RVV06,Sjo98}. In fact, for many dispersive PDE's the  convergence question is unresolved for $1/4 \leq s \leq 1/2$ (see \S \ref{sec_lit}). Our main result proves the first necessary condition on $s$ that is strictly above $1/4.$

\begin{thm}\label{thm_main_max}
Fix $n \geq 2$ and $k \geq 2$. Suppose there is a constant $C_s$ such that for all $f \in H^s(\R^n)$,
\beq\label{max_op} \|\sup_{0<t<1} |T_t^{(P_k)} f|\, \|_{L^1(B_n(0,1))}  \leq C_s \|f\|_{H^s(\R^n)}.\eeq
Then $s \geq \frac{1}{4} + \frac{n-1}{4((k-1)n+1)}.$
\end{thm}

  As a consequence of Theorem \ref{thm_main_max}, for $P=P_k$,  the convergence property (\ref{ptwiseP}) fails for all $s<  \frac{1}{4} + \frac{n-1}{4((k-1)n+1)}$ for each dimension $n \geq 2$.

An interesting open question remains: for $P=P_k$, what is the   value of $s_k^*(n)$  such that for all $s> s_k^*(n),$ the convergence property (\ref{ptwiseP}) holds for all $f \in H^s(\R^n)$, and for all $s< s_k^*(n)$ it fails? For $n=1,$ $s_k^*(n)=1/4$ for all $k \geq 2.$
 For $n\geq 2$ and degree $k=2$, $s_2^*(n)=n/(2(n+1))$ and our work recovers Bourgain's  construction. For degrees $k \geq 3$, the optimal value for $s_k^*(n)$ remains open in dimensions $n \geq 2$. We do not have a prediction for whether the threshold we obtain in Theorem \ref{thm_main_max} is optimal (but see Remark \ref{remark_dispersive}).
 
  Our results fit into a large body of research on dispersive PDE's, local smoothing, and maximal operators. In \S \ref{sec_lit}, we situate our results in that literature, which is rich with open problems. But first we describe our method for proving Theorem \ref{thm_main_max}, which is number-theoretic, and appears to be the first time  that the Weil bound has been introduced to study the regularity of solutions to a PDE. We anticipate our approach will be able to address many further open questions.  

\subsection{Method of proof}
We prove Theorem \ref{thm_main_max} by constructing an explicit family of counterexamples that violate the putative $H^s \maps L^1_{\mathrm{loc}}$ bound (\ref{max_op}) for every small $s$. 
\begin{thm}\label{thm_main}
 Fix $n \geq 2$ and $k \geq 2$.
Fix $s < \frac{1}{4} + \frac{n-1}{4((k-1)n+1)}.$ There exists a sequence of real numbers $R_j \maps \infty$ as $j \maps \infty$, and a sequence of functions $f_j \in L^2(\R^n)$ such that $\|f_j \|_{L^2(\R^n)} =1$ and $\hat{f_j}$ is supported in an annulus $\{(1/C)R_j \leq |\xi| \leq CR_j\},$ with the property that 
 \[ \lim_{j \maps \infty} R_j^{-s} \| \sup_{0<t<1} |T_t^{(P_k)} f_j| \, \|_{L^1(B_n(0,1))} = \infty.\]
\end{thm}
This immediately implies Theorem \ref{thm_main_max}, since if a function $f \in H^s(\R^n)$ is supported in such an annulus of radius $\approx R$, then $\|f\|_{H^s} \approx R^s \|f\|_{L^2};$ see \S \ref{sec_norm} for details.

To prove Theorem \ref{thm_main}, we construct for each large $R$ a   counterexample function  $f$ and a carefully chosen set  $\Omega^*$ of points $x \in B_n(0,1)$, such that for each   $x \in \Omega^*$, there is a choice of $t \in (0,1)$ for which $T_t^{(P_k)}f(x)$ can be well-approximated by an exponential sum of a certain length, which can in turn be well-approximated by a member of a family of ``large'' complete exponential sums modulo $q$, for primes $q$ in a well-chosen dyadic range. We then optimize the choices of all parameters in this construction---in particular, to ensure that simultaneously with all the above considerations, $\Omega^*$ has ``large'' measure in $B_n(0,1)$---and produce a counterexample  for each $s < \frac{1}{4} + \frac{n-1}{4((k-1)n+1)}$. From this bird's-eye-view, our method appears similar to Bourgain's  approach for the special case $P_2(\xi) = |\xi|^2$, which was rigorously  explained in the third author's work \cite{Pie20}.

But to succeed in a higher-degree setting, our method requires several completely new ideas. First, Bourgain's argument  relied only on Gauss sums, which can be evaluated by elementary methods \cite[Appendix B]{Pie20}. Exponential sums of higher degree polynomials are more complicated, and hence we require different methods to bound these sums from both above and below. We use an abstract argument in Proposition \ref{prop_sum_big} (and its corollaries) to show that ``most'' of the complete exponential sums we encounter are ``large.''  
 In particular, we capitalize on the Weil bound, which is a consequence of the truth of the Riemann Hypothesis over finite fields.

This argument shows that ``most'' sums are large, but does not identify \emph{which} sums are large. Consequently, this necessitates a much more abstract construction of the special set $\Omega^* \subset B_n(0,1)$ of points $x$ on which $\sup_{0<t<1}|T_t^{(P_k)}f(x)|$ can be shown to be large. We prove a very general result showing that if a collection of measurable sets is sufficiently well-distributed, then the measure of their union is comparable to the sum of the measures of the individual sets. We prove the general case in Lemma \ref{lemma_union_gen} and adapt it to our setting in Proposition \ref{prop_box_omega}. 
In particular, we use a number-theoretic argument to show that the boxes we construct are sufficiently well-distributed if they are centered at rationals with \emph{prime} denominators, yet another difference from the arguments in the quadratic case \cite{Bou16,Pie20}.  

These new strategies form a highly flexible framework, and we anticipate that they can be widely adapted to prove counterexamples for many maximal operators associated to dispersive PDE's.

\subsection{Related literature}\label{sec_lit}
Our work fits into a large family of questions about dispersive PDE's of the form 
\begin{align}
    \partial_t u - i \Pcal(D) u =0, \quad &(x,t) \in \R^n \times \R \label{PDE}\\
    u(x,0) =f(x), \quad &x \in \R^n, \nonumber
\end{align}
for a function $u$ acting on $\R^n \times \R$ and an initial data function $f$, where $D = \frac{1}{i}(\frac{\partial}{\partial x_1},\ldots,\frac{\partial}{\partial x_n})$, and $\Pcal(D)$ is defined according to its real symbol by 
\[ \Pcal(D)f(x) = \frac{1}{(2\pi)^n} \int_{\R^n} e^{i x \cdot \xi} P(\xi) \hat{f}(\xi) d\xi.\]
Roughly speaking, to be dispersive, the function $P$ must   behave like $|\xi|^\al$ for some $\al>1$   as $|\xi| \maps \infty$. (The case   $\al=1$ corresponds to the wave equation, which has   different behavior.) 
One main type (Schr\"odinger type) corresponds to the case that  $P(\xi) = q(|\xi|^2) $ for some appropriate function $q$; for example, $P(\xi) = |\xi|^\al$ for $\al >1$ corresponds to a power of the Laplacian, and in particular the case $\al=2$ leads to the linear Schr\"odinger equation. 
A second main type (Korteweg-de Vries type) corresponds to the case that $n=1$, $u$ is real-valued, and $P(\xi)= \xi q(\xi^2)$ for some appropriate function $q$; this includes the (free) KdV equation with $P(\xi) = \xi^3,$ the Benjamin-Ono equation with 
 $P(\xi) = \xi|\xi|$, the intermediate long-wave equation, Smith equation, and others (see  \cite{ConSau88}, and also a high-dimensional Benjamin-Ono equation in \cite{HLRRW19}).
More generally, nonlinear versions are also of interest, in which (\ref{PDE}) contains some further term $F(u)$ that is nonlinear in the function $u$; but strategies to prove results about the nonlinear case often rely on deductions involving the linear case, so that the linear case remains of central interest.  

We note that the precise  condition required of $P$ for the initial value problem (\ref{PDE}) to be considered dispersive can vary. One classical criterion appears in Constantin-Saut \cite[Eqns. (0.4)-(0.6)]{ConSau88}. A less restrictive criterion is developed by Kenig-Ponce-Vega \cite[Thm. 4.1]{KPV91}; our polynomial $P_k(\xi)$ satisfies the criterion of Kenig-Ponce-Vega.

A major focus in the study of the initial value problem (\ref{PDE}) is proving local smoothing; this refers to a phenomenon where the solution $u$ to the initial value problem is (locally) smoother than the initial data function $f$.
Quantitatively, for an equation of the form (\ref{PDE}) with $P(\xi)$ behaving (roughly) like $|\xi|^\al$ with $\al>1$ as $|\xi| \maps \infty$, local smoothing is a statement of the following form: if $f \in H^s(\R^n)$ then for a.e. $t \neq 0$, $u(\cdot,t) \in H_{\mathrm{loc}}^{s+\mu}(\R^n)$ where $\mu = (\al-1)/2$. Thus when $\al$ is larger, so that the differential equation (\ref{PDE}) is more dispersive, the local smoothing effect in the $x$ variable is stronger. (The effect is only local since  $e^{i tP(\xi)}$ has unit norm for every $t \in \R$, so that the solutions of the dispersive equation give a unitary group on the Sobolev space $H^s(\R^n) = W^{s,2}(\R^n);$ in particular for each $s$ the global $H^s(\R^n)$ norm of $u(\cdot,t)$ cannot differ from that of $f$.)
\begin{remark}\label{remark_dispersive}
Fix $n \geq 2$.
In Theorem \ref{thm_main_max} we prove that if (\ref{max_op}) holds, then $s \geq 1/4 + \del(k,n)$ for a value $\del(k,n)>0$ that decreases as $k$ increases. In our method, this decrease in $\del(k,n)$ for large $k$ is due to the tighter neighborhood we must impose on $t$ in (\ref{eqn: t condition 1 kpurediag}), (\ref{eqn: t condition 2 kpurediag}) as a result of needing to remove degree $k$ terms from the phase of an integral before applying Fourier inversion. This then imposes smaller neighborhoods in $y_1$ when we construct the set $x \in \Omega^*$ in (\ref{Omega_dfn}), and ultimately a tighter constraint on the parameter $Q$ in (\ref{conditions}). But it could be the truth that $\del(k,n)$ should decrease with $k$, given that the local smoothing effect increases as the dispersive effect increases with $k$.
\end{remark}

The literature on local smoothing (which also has connections to well-posedness, Strichartz estimates, and restriction theory) is far too vast to survey here, but we mention for example the influential works of Kato \cite{Kat83}, Constantin and Saut \cite{ConSau88}, and  Kenig, Ponce and Vega \cite{KPV91}. In particular, \cite[Theorem 4.1]{KPV91} proves local smoothing of the above type in a setting that includes the polynomial $P_k(\xi) = \sum_j \xi_j^k$ we study. 
The connection between space-time estimates for the Schr\"odinger operator and restriction theory is implicit in many of these articles, but for a few examples see e.g. Kenig, Ponce and Vega \cite{KPV91}, Moyua, Vargas, and Vega \cite{MVV96,MVV99}, Rogers \cite{Rog08}, and in particular the explicit connection derived in Lee, Rogers and Seeger \cite{LRS13}.

\subsubsection{Convergence results}\label{sec_lit_conv}
So far this has mentioned local smoothing with respect to $x$. Proving that for all $f \in H^s(\R^n)$, the pointwise convergence result (\ref{ptwiseP}) holds, additionally requires understanding regularity of $u(\cdot,t)$ in $t$. 
This is best understood when the symbol $P$ is of degree 2, including the important cases of the linear Schr\"odinger equation and the non-elliptic Schr\"odinger equation. 
For all higher-degree symbols,  the previous literature left open the convergence question for $1/4 \leq s \leq 1/2$ in  dimensions $n \geq 2$. 
For clarity, we  briefly give specific citations, to highlight the context of breaking the 1/4 barrier in Theorem \ref{thm_main_max}.

(a) The symbol  $P_2(\xi) = |\xi|^2$: in this case (\ref{PDE}) is the linear Schr\"odinger equation. For $n=1,$ (\ref{ptwiseP}) holds if and only if $s \geq 1/4$ by Carleson \cite{Car80} and Dahlberg-Kenig \cite{DahKen82}.
The convergence question in dimensions $n \geq 2$ has a long history, including works by Carbery \cite{Car85}, Cowling \cite{Cow83}, Sj\"olin \cite{Sjo87}, Vega \cite{Veg88}, Bourgain \cite{Bou95}, Moyua-Vargas-Vega \cite{MVV96}, Tao-Vargas \cite{TaoVar00}, Lee \cite{Lee06}, Bourgain \cite{Bou13}, Luc\`a-Rogers \cite{LucRog17}, Demeter-Guo \cite{DemGuo16}, Bourgain \cite{Bou16}, Luc\`a-Rogers \cite{LucRog19}, Du-Guth-Li \cite{DGL17}, and Du-Guth-Li-Zhang \cite{DGLZ18}.
 Bourgain \cite{Bou16} and Du-Zhang \cite{DuZha19} resolved the question (up to the endpoint): (\ref{ptwiseP}) holds if $s > s_2^*(n) = n/(2(n+1))$ and fails if $s< s_2^*(n)$.

(b) The symbol $P_2^-(\xi) :=\xi_1^2 - \xi_2^2 \pm \xi_3^2 \pm \cdots \pm \xi_n^2:$ 
in this case (\ref{PDE}) is the  non-elliptic Schr\"odinger equation (and when $P_2^-$ has only one change of sign, $\Pcal(D)$ is the box operator $\square$). 
Then by Rogers, Vargas, and Vega \cite{RVV06}, for all $n \geq 2$ (\ref{ptwiseP}) fails if $s<1/2$ and holds for all $s>1/2$. In that work they also note that for $n=2$, (\ref{ptwiseP}) holds for $s=1/2$, due to an observation of E. M. Stein (see the proof in \cite[p. 1900]{RVV06}).
It is interesting that this behavior differs significantly from the elliptic case. 

(c) The symbol $P: \R^n \maps \R$ is a polynomial of degree $2$:
 for any $n \geq 1$,   (\ref{ptwiseP}) holds for all $s> 1/2$ by Rogers, Vargas and Vega \cite[Thm. 1.2 and 2.1]{RVV06} (recording a method of \cite{Veg88b}). It fails for $s < 1/4$ by several methods, e.g. Sj\"olin \cite{Sjo98} (see Remark \ref{remark_Sjo}).
 Aside from the special cases $n=1$ or $P_2(\xi)$ and $P_2^-(\xi)$ for $n \geq 2$, the convergence question is open for $1/4 \leq s \leq 1/2$.

(d) The symbol $P(\xi) =|\xi|^\al$ for real $\al>1$ on $\R^n$:  Sj\"olin \cite[Thms. 2-5]{Sjo87} proved that for $n=1,2$ (\ref{ptwiseP})  holds  if $s \geq n/4$ and for $n \geq 3$ it holds if $s> 1/2$. For all $n \geq 1$ it fails if $s < 1/4$. This is also proved in Vega \cite[Thm. 1']{Veg88}; similar estimates also appear in Constantin and Saut \cite{ConSau88}. In dimensions $n \geq 2$,  for $P(\xi) = |\xi|^\al$  with $\al >1$, $ \al \neq 2$, the convergence question is open  for $1/4 \leq s \leq 1/2$ (or $1/4 \leq s < 1/2$ if $n=2$). 

(e) The symbol $P: \R^n \maps \R$ is a polynomial of degree $k \geq 2$:
this is the case in which our polynomial $P_k(\xi)$ lies.
In the case $n=1$,  Kenig, Ponce and Vega \cite[Cor. 2.6]{KPV91} prove that (\ref{ptwiseP}) holds for any polynomial $P$ of degree $k \geq 2$ and $s>1/4$ (and even when $P$ is replaced by $R((\xi))^\al$ with $\al \neq 0$, for a rational function $R$); it fails if $s<1/4$ by \cite{DahKen82}, \cite{KenRui83}. 
 For all $n \geq 1$, Ben-Artzi and Devinatz \cite[Thm. D]{BenDev91}  prove (\ref{ptwiseP}) for all $s>1/2$, for any real polynomial $P$ of principal type of order $\al$ for $\al>1$ (that is, such that $|\nabla P(\xi)| \gg (1+|\xi|)^{\al-1}$ for all sufficiently large $|\xi|$).
Furthermore, Rogers, Vargas, Vega prove (\ref{ptwiseP}) holds for all $s>1/2$ if $P: \R^n \maps \R$ is a member of an appropriate class of differentiable functions, which in particular includes  polynomials \cite[Remark 2.2]{RVV06}. 
The convergence property (\ref{ptwiseP}) fails for $s < 1/4$ by several methods, including e.g. Sj\"olin \cite{Sjo98} (see Remark \ref{remark_Sjo}).
 This left the convergence question in the range $1/4 \leq s \leq 1/2$ open, until the present paper.

We further remark that  convergence problems like (\ref{ptwise}) and (\ref{ptwiseP}) are also being studied from many more perspectives. For example: in relation to non-tangential convergence \cite{SjoSjo89}; when $t$ varies in a set defined according to a complex parameter \cite{Sjo09}, \cite{SjoSor10}; convergence along restricted directions or variables curves \cite{CLV12}; and along $t$ belonging to various types of countable sequences $\{t_n\}$ of points in $(0,1)$, see e.g. \cite{Sjo19JFAA}, \cite{SjoStr20}, \cite{DimSee20}, \cite{SjoStr21},
or certain uncountable sets \cite{SjoStr20}.
There are also interesting recent studies related to pointwise convergence of solutions of analogous PDE's in other settings: on the torus \cite[\S 5]{KPV91}, \cite{Veg92}, \cite{MoyVeg08}; on manifolds \cite{WanZha19}; and for the nonlinear Schr\"odinger flow \cite{CLS21}.

Finally, we note that Bourgain's work on counterexamples for the Schr\"odinger maximal operator associated to
 $P_2(\xi) = |\xi|^2$  stimulated a number of new works for quadratic symbols, including the study of divergence on sets of lower-dimensional Hausdorff measure, for example in
\cite{BBCR11}, \cite{LucRog19a}, \cite{LucRog19}, \cite{LucPon21}, and the study of the failure of local maximal estimates analogous to (\ref{max}) in higher $L^p$ spaces  \cite{DKWZ20}. We anticipate that the methods of the present paper can be adapted to study higher-degree analogues of many such questions.

 \subsubsection{Maximal operators}\label{sec_lit_max}
Our main theorem is a statement about a maximal operator. This is closely related to the literature on convergence results mentioned above, since by classical arguments,  appropriate maximal estimates imply convergence results, and the failure of certain maximal estimates implies the failure of certain convergence results (see e.g. \cite[Appendix A]{Pie20}).
But additionally, there is a broad literature on maximal operators in their own right.

For any given symbol $P$,  several types of maximal estimates are typically studied. One can study the maximal operator $\sup_{t \in I} |T_t^{(P)}f|$ for a bounded interval $I$, or for an infinite interval $I$; one can ask whether a local norm of this operator is bounded, or a global norm; one can consider $f$ in an $L^2$ Sobolev space such as $H^s =W^{s,2}$ or in an $L^q$ Sobolev space $W^{s,q}$ for $q>2.$ Further  questions study not the maximal operator in $t$, but $L^q$-means over $t$ in some compact interval $I$. 

Our result, for a \emph{local} $L^1$ norm of the maximal operator with $t$ in the \emph{bounded} interval  $(0,1)$, and with initial data $f \in H^s(\R^n)$,   is a strong result in the hierarchy of types described above. In particular, for a given $s$, our result that (\ref{max_op}) \emph{fails} implies that the corresponding inequality must also fail for the (larger) maximal operator over $0< t <\infty$; for the (larger) global $L^1(\R^n)$ norm; for the (larger) local $L^p(B_n(0,1))$ norm for all $p >1$.

To situate our results in the large literature on maximal operators, we highlight here a few of the most relevant papers.

(a) The paper that lies closest to our maximal estimates for   $P_k(\xi) = \sum_j \xi_j^k$ with $k \geq 3$ is by Sj\"olin \cite{Sjo98}. Sj\"olin studies local $L^q(B_n(0,1))$ norms for the operator $f \mapsto \sup_{0<t<1}|T_t^{(\Phi)}f(x)|$ 
for functions $\Phi(\xi) = \phi_1(\xi) + \cdots + \phi_\ell(\xi),$
where each $\phi_j$ is a real-valued $\mathcal{C}^2(\R^n \setminus\{0\})$ function that is homogeneous of degree $a_j$, where $0 \leq a_1< a_2< \cdots < a_{\ell-1} \leq a_\ell -1/2$. Under the assumption that $a :=a_\ell \geq 1$ and $\phi_\ell$ does not vanish identically, Sj\"olin proves that if 
\beq\label{maxSjo}
\| \sup_{0<t<1} |T_t^{(\Phi)}f| \|_{L^q(B_n(0,1))} \ll_{q,s} \|f\|_{H^s(\R^n)} 
\eeq
for all $f \in H^s(\R^n)$ then
\beq\label{condSjo} s  \geq \frac{n}{4}- \frac{n-1}{2q}.
\eeq
If $n=1$ then this is the requirement $s \geq 1/4$. In the setting of our Theorem \ref{thm_main_max} ($n\geq 2$ and $q=1$), the condition (\ref{condSjo}) provides no nontrivial lower bound on $s>0$. 

\begin{remark}\label{remark_Sjo} We recall that standard arguments show that for a fixed real symbol $P$, if the convergence result (\ref{ptwiseP}) holds for all $f \in H^s(\R^n)$, then (\ref{maxSjo}) must hold for $q=2$ (see \cite[Appendix A]{Pie20}).  
Sj\"olin's result (\ref{condSjo}) at $q=2$ confirms that for all symbols covered by his methods, for the convergence result to hold, $s \geq 1/4$ is necessary, for all dimensions $n \geq 1$. (We also recover this; see Remark \ref{remark_quarter}.) 
\end{remark}

 (b) Many works have considered the maximal operator associated to $P_2(\xi) = |\xi|^2$ in terms of global $L^q(\R^n)$ norms for various $q$, for $f \in H^s$.
This has been studied both for the local case $0<t<1$ and the global case $t \in \R$ for $q=2$ \cite{Sjo94}, and for other $q$ by Sj\"olin \cite{Sjo97}; sharp results for $q \neq 2$ are obtained in Rogers and Villarroya \cite{RogVil07}.
See Sj\"olin \cite{Sjo07} and Rogers, Vargas, and Vega \cite{RVV06} for the non-elliptic case $P_2^-(\xi),$ for $f \in H^s$ and various $L^q(\R^n)$ global norms, corresponding to case (b) in \S \ref{sec_lit_conv}.
The equivalent problems with homogeneous Sobolev spaces $\dot{H}^s$ have been studied as well; see e.g. \cite{Sjo02}.

(c) For the case $P(\xi) = |\xi|^a$ with $a>1$,
Sj\"olin has  characterized for which $s, q$ (\ref{max_op}) can hold when $n=1$, in terms of the local supremum and local $L^q$ norm \cite{Sjo97}.
For $n\geq 2$, maximal operators over $0<t<1$  have been extensively studied for both the local and global norm,  for $f$ radial  \cite{Sjo95,Sjo97,Sjo11}; see also Wang \cite{Wan97, Wan06}.
More recently, Sj\"olin \cite{Sjo07} has also considered the case $P(\xi) = \sum_j \pm |\xi_j|^{\al}$, $\al>1$, but for $t\in \R$ and for $L^q(\R^n)$ global norms, which have quite a different flavor from our result, since for such global norms, standard homogeneity arguments place tight restrictions on $q$ relative to $n,s$ (see e.g. arguments in \cite[\S 2.4]{Sjo07}).

There are many further investigations of maximal operators that generalize the Schr\"odinger setting in other ways. For example:
bounds for multiparameter analogues of maximal Schr\"odinger operators are considered in \cite{SjoSor14};  Rogers and Villarroya \cite{RogVil08} prove sharp results for the maximal operator associated to the wave equation;
  bounds of $L^q$-means with respect to $t$ (rather than a supremum over $t$), are studied for example by Rogers and Rogers--Seeger \cite{Rog08}, \cite{RogSee10}.
Finally, of course, many of the convergence results mentioned in the previous section are in fact stated in terms of results for maximal operators.

\subsection{Outline of the paper}
 In Section \ref{sec_exp_sums} we prove upper and lower bounds for exponential sums, which are the critical ingredients to force  $\sup_{0<t<1}|T_t^{(P_k)}f(x)|$ to be large for many values  of $x$.  In Section \ref{sec_reduction} we construct a family of functions $f$, according to certain parameters,   and reduce the study of $T_t^{(P_k)}f(x)$ to an exponential sum. In Section \ref{sec_Omega} we motivate our choice for the set $\Omega^*$ of points $x$ for which there exists a choice of $0<t<1$ such that $|T_t^{(P_k)}f(x)|$ is large. We then give an abstract proof to show that the measure of $\Omega^*$ is sufficiently large. In Section \ref{sec_arithmetic} we use our explicit choice of the set $\Omega^*$ to evaluate the exponential sum and bound the error terms. Finally, in Section \ref{sec_choices} we assemble all of these constructions and make the optimal choices of parameters that prove Theorem \ref{thm_main} and hence Theorem \ref{thm_main_max}.

 \subsection{Notation}

 We use the  notation $A\ll B$ to indicate that there is a constant $C$ such that $|A| \leq C B$; the notation $A \ll_\al B$ indicates that the constant $C$ may depend on the parameter $\al$. In general, we will allow implicit constants to depend on  the dimension $n$, the degree $k$, and a $C^\infty$ function  $\phi$ that is fixed once and for all. We will denote certain small constants we can freely choose by $c_1, c_2, c_3,\ldots$. Since we will use our ability to choose them to our advantage, we will denote them explicitly in inequalities in which their small size plays a role.

We let $B_m(c,r)\subseteq \R^m$ denote the Euclidean ball centered at $c$ and of radius $r$ and let $A_m(R,C)\subseteq \R^m$ denote the annulus $B_m(0,CR)\backslash B_m(0,R/C)$. 
 For a finite set $\Acal$ we let $|\Acal|$ denote its cardinality; for a   Lebesgue measurable set $\Omega$, we let $|\Omega|$ denote its Lebesgue measure.

We follow the convention in \cite{Bou16} of letting $e(t) = e^{it}.$ Correspondingly, we use the normalization for the Fourier transform that   $\hat{f}(\xi) = \int_{\R^m} f(x) e^{-i x \cdot \xi} d x$ and  $f(x) = (2\pi)^{-m}\int_{\R^m}\hat{f}(\xi)e^{i x \cdot\xi} d\xi$. Then Plancherel's theorem states  that $\|f\|^2_{L^2(\R^m)} = (2\pi)^{-m}\|\hat{f}\|^2_{L^2(\R^m)}$. 
The Sobolev space $H^s(\R^n)$ is the set of $f \in \Scal'(\R^n)$ with finite Sobolev norm
\beq\label{Sobolev}
\|f\|_{H^s(\R^n)}^2 = \frac{1}{(2\pi)^n}\int_{\R^n} (1+|\xi|^2)^{s}|\hat{f}(\xi)|^2 d\xi.
\eeq

\section{Upper and lower bounds for exponential sums}\label{sec_exp_sums}
We will use complete exponential sums  to show that for each $x$ in a certain set $\Omega^* \subset B_n(0,1),$ we can choose a $t\in (0,1)$ to make $|T_t^{(P_k)} f(x)|$ large.
 We first prove Proposition \ref{prop_sum_big}, which shows that for each prime $q$, a sufficiently numerous collection of complete exponential sums modulo $q$ is large. 
Second, in order to bound the contribution of certain error terms from above, we  develop an upper bound for exponential sums in Proposition \ref{prop_sum_approx}.
Both our lower and upper bounds rely on the Weil bound, which we now state.

The Weil bound is a consequence of Deligne's proof of the Weil conjectures \cite{Del74}; we  cite this in the form provided by \cite[Thm. 11.43]{IwaKow04};   $\mathrm{Tr}$ denotes the trace function from $\F_{q^n}$ to $\F_q$.
 \begin{lemma}[Weil bound]
  Let $f \in \Z[X_1,\ldots, X_m]$ be a nonzero polynomial of degree $k$ such that the hypersurface $H_f$ in $\mathbb{P}^{m-1}$ defined by the equation 
$H_f: f_k(x_1,\ldots,x_m)=0$
 is nonsingular, where $f_k$ is the homogeneous component of $f$ of degree $k$. For any prime $q \ndiv k$ such that the reduction of $H_f$ modulo $q$ is smooth, any nontrivial additive character $\psi$ modulo $q$ and any $n \geq 1$, 
  \[ \Bigg| \sum_{x_1, \ldots, x_m \in \F_{q^n}} \psi (\mathrm{Tr} (f(x_1,\ldots,x_m))) \Bigg| \leq (k-1)^m q^{nm/2}.\]
  \end{lemma}

 We apply this in the case of $q$ prime,  $q \ndiv k$, $m=1$,  $n=1$, and $f = P\in \Z[X]$  a polynomial of degree $k \geq 2$, with leading coefficient $c_k$. Then  $f_k(x) = c_k x^k$, where we assume that $q \ndiv c_k$, so that $H_f$ is nonsingular. Then the Weil bound is
 \beq\label{Weil}
 \Bigg|\sum_{x \modd{q}} e(2\pi  P(x)/q) \Bigg|\leq (k-1) q^{1/2}.\eeq

\subsection{Large values of  exponential sums with rational coefficients}

In what follows, we use the convention $\underline{u} \modd{q}^s$ to indicate that in each coordinate of $\underline{u} = (u_1,\ldots,u_s)$,  $u_i$ runs modulo $q$.

  We now show that a positive proportion of choices for integral coefficients  lead to a complete exponential sum modulo $q$ of size $\gg q^{1/2}$ (which by (\ref{Weil}) is optimal, up to a constant).

\begin{prop}\label{prop_sum_big}
 Fix an integer $s \geq 2$ and  integers $1 = k_s  < \cdots < k_2 < k_1$.
For each integer $q$ and tuple $\underline{a} = (a_1,\ldots, a_s)$ let
\[T(\underline{a};q) =T(a_1,\ldots,a_s;q) := \sum_{n \modd{q}} e \bigg((a_1 n^{k_1} + \cdots + a_s n^{k_s})\frac{2\pi }{q} \bigg).\]
Then there exist constants $0<\al_1<1$ and $0<\al_2<1$ with $\al_2$ depending on $k_1$, such that for every prime $q \geq 3$ with $q \ndiv k_i$ for all $i \in \{1,\dots,s\}$, at least $\al_2 q^s$ choices of $\underline{a} \modd{q}^s$ have $|T(\underline{a};q)| \geq \al_1 q^{1/2}$. In fact, one can take $\al_1=1/2$ and $\al_2 = k_1^{-2}/4$.
\end{prop}

In the case that the exponents $k_s,k_{s-1},\ldots, k_1$ are $1, 2 ,\ldots, s$ (respectively), this is Theorem 14 of \cite{KniSok79}, also recorded as \cite[Lemma 2.4]{CheShp20}. The case of sparse exponents   is remarked upon in \cite[\S 6.1]{CheShp20}, and we thank Igor Shparlinski for communicating their method of proof. 
\begin{proof}
Parseval's theorem shows that 
\beq\label{S_id}
 \sum_{ \underline{a} \modd{q}^s} |T(\underline{a};q)|^2 = q^{s+1}. 
 \eeq
Indeed, expanding the left-hand side as 
\[ \sum_{\underline{a}  \modd{q}^s} \sum_{n \modd{q}} e \bigg( (a_1 n^{k_1} + \cdots + a_s n^{k_s})\frac{2\pi }{q}\bigg)\sum_{m \modd{q}} e \bigg(-(a_1 m^{k_1} + \cdots + a_s m^{k_s})\frac{2\pi }{q}\bigg)\]
and summing first over $a_1, \ldots, a_s \modd{q}$ we gain a contribution of $q^{s+1}$ precisely for those $n,m$ with $n\con  m \modd{q}$, which confirms the claim.

Suppose now that for certain constants $\al_1,\al_2>0$ (to be specified later), there are  $<\al_2 q^s$ choices of $\underline{a} \modd{q}^s$ with 
$|T(\underline{a};q)| \geq \al_1 q^{1/2}$. Then write
\[
 \sum_{\underline{a} \modd{q}^s} |T(\underline{a};q)|^2 
= q^2+ \sum_{\bstack{\underline{a}  \modd{q}^s, \underline{a}  \not\con 0 }{|T(\underline{a};q)| \geq \al_1 q^{1/2}}} |T(\underline{a};q)|^2 + \sum_{\bstack{\underline{a}  \modd{q}^s, \underline{a}  \not\con 0 }{|T(\underline{a};q)| < \al_1 q^{1/2}}} |T(\underline{a};q)|^2  \nonumber .\]
 Here the first term on the right-hand side is from $\underline{a}  \con 0 \modd{q}^s$, in which case $|T(\underline{a} ;q)|^2 = q^2$. In the second term, we can apply our assumption  to bound the number of values of $\underline{a}$ included in the sum, and then apply the Weil bound to each term $|T(\underline{a};q)|$.  
In particular the Weil bound  (\ref{Weil})   shows that if $k_i$ is the largest exponent for which $a_{i} \not\con 0 \modd{q}$, then $|T(\underline{a};q)| \leq (k_i-1) q^{1/2} \leq k_1 q^{1/2}$.  
Since the smallest exponent is $k_s =1$, this observation applies for all $\underline{a} \not\con 0 \modd{q}^s$ with some nonzero coefficient of a nonlinear term, that is, with $a_i \not\con 0 \modd{q}$ for some $i \in \{1,2,\ldots,s-1\}.$ 
For the remaining cases of $\underline{a} \not\con 0$, namely   $\underline{a} = (0,0,\ldots,a_s)$ with $a_s \not\con 0 \modd{q}$, we observe that
 \[ T(\underline{a};q) = \sum_{n \modd{q}} e(2\pi a_sn/q)=0,\]
 which also suffices.
We conclude that
 \[
 \sum_{\underline{a} \modd{q}^s} |T(\underline{a};q)|^2 
 < q^2 +  \al_2q^s ( k_1 q^{1/2})^2   + q^s (\al_1q^{1/2} )^2  \leq  (1/3 +\al_2 k_1^2 + \al_1^2)q^{s+1}.\]
Here we used the fact that if $q \geq 3$ then $q^2 \leq (1/3) q^{s+1}$ for all $s \geq 2$. 
Now we see that for any $\al_1,\al_2$ small enough that $(1/3 +\al_2 k_1^2 + \al_1^2)<1$,    we have obtained a contradiction to the identity (\ref{S_id}). Thus for any sufficiently small choices of $\al_1, \al_2$, 
there are $ \geq \al_2 q^s$ values of $\underline{a} \modd{q}^s$ such that $|T(\underline{a};q)| \geq \al_1 q^{1/2}$. In particular we may take  $\al_1=1/2$ and $\al_2 = k_1^{-2}/4$.  
 \end{proof}

In order to apply Proposition \ref{prop_sum_big}, we need the following corollary, which distinguishes the role of the highest-order coefficient. Here we let $\underline{a} = (a_1,a')$ with  $a_1 \in \Z$ and $a'\in \Z^{s-1}.$
\begin{cor}\label{cor_T_big}
Fix an integer $s \geq 2$ and  integers $1 = k_s  < \cdots < k_2 < k_1$. For each prime $q$ let $T(\underline{a};q) = T((a_1,a');q)$, and  specify $\al_1,\al_2$ to be as in  Proposition \ref{prop_sum_big}. 
Let $\Acal(q)$ denote the set of $\underline{a} \modd{q}^s$ such that $|T(\underline{a};q)| \geq \al_1q^{1/2}.$ 
For each $a_1 \modd{q},$ define the ``good set'' 
\[G(a_1) := \{ a' \modd{q}^{s-1}: (a_1,a') \in \Acal(q)\}.\] 
Suppose that $q\ge 3$ is a prime such that $q \ndiv k_i$ for all  $i \in \{1,\dots,s\}$.
Then for at least $(\al_2/2) q$ choices of $a_1 \modd{q},$ we have $|G(a_1)| \geq (\al_2/2) q^{s-1}.$
 
\end{cor}
\begin{proof}

 For a fixed $0<\al_3<1$ to be determined later, write
\[ |\Acal(q)|  =  \sum_{a_1 \modd{q}} |G(a_1)| = \sum_{\bstack{a_1\modd{q}}{|G(a_1)| \geq \al_3q^{s-1}}} |G(a_1)|
    + \sum_{\bstack{a_1\modd{q}}{|G(a_1)| < \al_3q^{s-1}}} |G(a_1)|.\]
 By Proposition \ref{prop_sum_big} we know that $|\Acal(q)| \geq \al_2q^s.$ We can bound the last term on the right-hand side from above by $q\cdot \al_3 q^{s-1}.$  On the other hand, we always know that $|G(a_1)| \leq q^{s-1}$.  Thus after rearranging, we see that 
 \[   \sum_{\bstack{a_1\modd{q}}{|G(a_1)| \geq \al_3q^{s-1}}} q^{s-1} \geq  \sum_{\bstack{a_1\modd{q}}{|G(a_1)| \geq \al_3q^{s-1}}} |G(a_1)|  \geq \al_2q^s - \al_3 q^s.\]
 Upon taking (for example) $\al_3=\al_2/2$, we learn that 
\[|\{ a_1 \modd{q} : |G(a_1)| \geq (\al_2/2)q^{s-1}\}| \geq (\al_2/2)q.\]
 
 \end{proof}

Finally, we are actually interested in products of sums of the form $T(\underline{a};q)$.  Here we replace the notation  $\underline{a}$ by $(a_1,\underline{b})$ where $a_1 \in \Z$ and $\underline{b} \in \Z^{s-1}.$ For any $(a_1,\underline{b}) \in \Z \times \Z^{s-1}$ define 
 $T((a_1,\underline{b});q)$ as in Proposition \ref{prop_sum_big}.  We are interested in products of $m$ copies of these sums, where $a_1$ varies $\modd{q}$ and each of $\underline{b}_1, \ldots, \underline{b}_m$ varies $\modd{q}^{s-1}.$
 
\begin{cor}\label{cor_T_prod_big_0}
Fix an integer $s \geq 2$ and  integers $1 = k_s  < \cdots < k_2 < k_1$.  Let $\al_1,\al_2$ and $T((a_1,\underline{b});q)$   be as in  Proposition \ref{prop_sum_big}.
Let $\Gcal(q) \subset \F_q \times \F_q^{s-1} \times \cdots \times \F_q^{s-1}$ denote the set of $a_1 \modd{q}$ and $\underline{b}_1, \ldots, \underline{b}_m \modd{q}^{s-1}$ for which 
\beq\label{T_big_ineq_0}
|T((a_1,\underline{b}_1);q)| \cdots |T((a_1,\underline{b}_m);q)| \geq \al_1^mq^{m/2}.\eeq
Then for every prime $q \ge 3$   such that $q \ndiv k_i$ for all $i \in \{1,\dots,s\}$,
we have
\[ |\Gcal(q)| \geq (\al_2/2)^{m+1} q^{1 + m(s-1)}.\]
\end{cor}
\begin{proof}
For each $a_1 \modd{q},$ define the ``good set'' 
$G(a_1)$ as in Corollary \ref{cor_T_big}.  
Now define
\[A_1(q) := \{ a_1 \modd{q}: |G(a_1)| \geq (\al_2/2) q^{s-1}\}.\]
By Corollary \ref{cor_T_big}, $|A_1(q)| \geq (\al_2/2)q$.  Now for each $a_1 \in A_1(q)$ we let $\underline{b}_1, \ldots, \underline{b}_m$ vary independently over $G(a_1)$. This gives us a collection of 
\[\geq (\al_2/2)q \cdot (\al_2/2)q^{s-1} \cdots (\al_2/2)q^{s-1}\]
tuples $(a_1,\underline{b}_1, \ldots, \underline{b}_m) \in \F_q \times \F_q^{s-1} \times \cdots \times \F_q^{s-1} $ for which $|T((a_1,\underline{b}_i);q)| \geq \al_1q^{1/2}$ for each $1 \leq i \leq m$, so that (\ref{T_big_ineq_0}) holds. 
\end{proof}

 When we construct the set $\Omega^*$ in Section \ref{sec_Omega}, we will apply Corollary \ref{cor_T_prod_big_0} in the case of $s=2$, $m=n-1,$ with $k_1=k \geq 2$ and $k_2=1.$
 In this case we will also need a good \emph{upper} bound for the product of the $T(a_1,b;q)$, so we modify the sets $\Gcal(q)$ by removing one point.
 \begin{cor}\label{cor_T_prod_big}
Fix an integer $k \geq 2$ and set $k_1=k, k_2=1$ in Corollary \ref{cor_T_prod_big_0}, so that
$\al_1=1/2,\al_2=k^{-2}/4$. 
Let $\Gcal^*(q) \subset \F_q^* \times \F_q^{n-1}$ denote the set of $a_1, a_2,\ldots, a_n \modd{q}$   for which
\beq\label{T_big_ineq}
\al_1^{n-1}q^{(n-1)/2} \leq |T((a_1,a_2);q)| \cdots |T((a_1,a_n);q)| \leq  (k-1)^{n-1} q^{(n-1)/2}.\eeq
Then for every prime $q$   such that $q  \geq 16k^2$,
\[ |\Gcal^*(q)| \geq (\al_2/2)^{n}(1-2^{-n}) q^{n}.\]
\end{cor}
We will use these sets $\Gcal^*(q)$ to pick the rationals $a_1/q, a_2/q,\ldots, a_n/q$ we use to build the set $\Omega^*$ in Section \ref{sec_Omega}.
\begin{proof} 
If $q>k$ and $q \ndiv a_1$ then we can apply the Weil bound to each factor, and prove the upper bound in (\ref{T_big_ineq}). If $q | a_1$ then $T(a_1,b;q)=0$ unless $q|b$ also. In particular, the only choice of $(a_1,a_2,\ldots, a_n)$ with $q |a_1$ that could contribute to the large values in (\ref{T_big_ineq}) must have $(a_1, a_2,\ldots, a_n) \con (0,0,\ldots, 0) \modd{q}$, in which case the product in (\ref{T_big_ineq}) is actually of size $q^{n-1}$. 
Taking $\Gcal(q)$ as in Corollary \ref{cor_T_prod_big_0} we now set
\beq\label{Gstar}
\Gcal^*(q) = \Gcal(q) \setminus 0.
\eeq
Then for every prime $q>k$, (\ref{T_big_ineq}) holds for every $(a_1,a_2,\dots, a_n) \in \Gcal^*(q)$.   Moreover, as long as we assume that $q \geq (\al_2/4)^{-1} = 16k^2,$ then 
\beq\label{Gstar_size}
|\Gcal^*(q)| \geq |\Gcal(q)| -1 
\geq ((\al_2/2)^{n}-q^{-n})q^{n} \geq  (\al_2/2)^{n}(1-2^{-n})q^{n}.
\eeq
Finally, we record that for all $\underline{a} \in \Gcal^*(q),$ $q \ndiv a_1.$
\end{proof}


\subsection{Upper bounds for incomplete exponential sums}
We require upper bounds for incomplete sums modulo $q$. We prove that an incomplete sum is at most as large (up to a logarithmic factor) as the complete sum, by the standard method of completing the sum and applying   the Weil bound.

\begin{lemma} \label{upper_bound_incomplete}
 Let $P \in \Z[X]$ be a polynomial of degree $k \geq 2$, with leading coefficient $c_k$. Let $q$ be a prime with $q \ndiv k$, $q \ndiv c_k$.
 Then for every $1 \leq H \leq q,$
 \[ \Bigg|\sum_{1 \leq n \leq H} e(2\pi P(n)/q) \Bigg|\ll_k q^{1/2} (\log q).\]
 \end{lemma}
 
\begin{proof}
Let $S(H)$ denote the sum on the left-hand side.
For any integer $q$, we can complete the sum by writing 
\begin{eqnarray*}
    S(H) &=& \sum_{1 \leq a \leq q}e(2\pi P(a)/q)  \sum_{\bstack{1 \leq n \leq H}{n \con a \modd{q}}} 1\\
        & = & \sum_{1 \leq a \leq q} e(2\pi P(a)/q) \sum_{1 \leq n \leq H} \frac{1}{q} \sum_{1 \leq h \leq q} e(2\pi (h(n-a)/q))) \\
        & = & \frac{1}{q} \sum_{1 \leq h \leq q}  \sum_{1 \leq a \leq q} e(2\pi (P(a) -ha)/q)
            \sum_{1 \leq n \leq H} e(2\pi hn/q).
\end{eqnarray*} 
Now in the case that $q$ is a prime (our case of interest), under the assumption that $q \ndiv c_k$, we can apply the Weil bound (\ref{Weil}) to the sum over $a$, to achieve the bound $\leq (k-1)q^{1/2}$ for the absolute value of this sum, uniformly in $h$. We also recall that the geometric series summed over $1 \leq n\leq H$ is bounded by $\ll \min \{ H, \|h/q\|^{-1}\}$, in which $\|t\|$ denotes the distance from $t$ to the nearest integer. By separating   into the cases $h \leq q/2$ and $q/2< h \leq q$ we see that the sum of $\min\{ H, \|h/q\|^{-1}\}$ over $1 \leq h \leq q$ is bounded by $\ll q \log q.$ We can conclude that 
\beq\label{completed}
|S(H)| \ll (k-1) q^{1/2}(\log q).
\eeq
\end{proof}

 \subsection{A sum in which the top coefficient is rational}
 
 We will also encounter exponential sums 
  \[ \sum_{M<n \leq M+N} e(2\pi P(n)) \]
  in which the degree $k$ polynomial $P$ has real coefficients.   Here it would be standard to apply the Weyl bound, which for $k \geq 3$ is substantially weaker than the square-root cancellation bounds we have seen above.
  (Here is another difference from Bourgain's quadratic case; for $k=2$ the Weyl bound only differs from square-root cancellation by a logarithm.)
 Following the resolution of the Main Conjecture for the Vinogradov Mean Value Theorem \cite{BDG16} (see also \cite{Woo16c,Woo19}), an improvement on the classical Weyl bound is now available. For $k \geq 3$, if for some $2 \leq j \leq k$ the coefficient $\al_j$ satisfies $|\al_j  - a/q| \leq 1/q^2$ for some $(a,q)=1$, then
\[ \Bigg| \sum_{M< n \leq M+N} e(2\pi P(n)) \Bigg| \ll N^{1+\ep} (q^{-1} + N^{-1} + qN^{-j})^{\frac{1}{k(k-1)}}.\]
 This is recorded by Bourgain \cite[Thm. 5]{Bou17b} but was already known to be an outcome of proving the Main Conjecture; see \cite[Ch. 4 \S 1]{Mon94}.
 But a direct application of this bound in our setting would significantly weaken our method.

Instead we use the following critical observation that is specific to the counterexample we construct. The exponential sums  we encounter are determined by a particular choice of $x \in \Omega^* \subset B_n(0,1)$ and $t \in (0,1)$. The fact that for each $x$ in the set $\Omega^*$ we construct we can \emph{choose} $t$ allows us to ensure that the leading coefficient of our degree $k$ polynomial  is \emph{rational}. This means we only accrue an error term by approximating the linear term; this is very advantageous.
Here is the core estimate we require.

\begin{prop}\label{prop_sum_approx}
Suppose that $P(n) = 2\pi a_1n^k/q + yn$ for an integer $k \geq 2$, a prime $q \ge 3$ such that $q \ndiv k$, an integer $1 \leq a_1<q$, and a real value $y$. Suppose also that $|y - 2\pi b/q| \leq V,$ for some integer $1 \leq b \leq q$ and some real  $V \geq 0$. Then for every $N \geq 1,$
 \beq\label{ST_sum}
 \Bigg|\sum_{M<n \leq M+N} e(P(n))\Bigg|
  = \lfloor N/q \rfloor \cdot \Bigg|\sum_{1 \leq n \leq q} e(2\pi (a_1n^k + bn)/q)\Bigg|  +E\eeq
  in which 
 \[ |E| \ll_k NV(\lfloor N/q \rfloor q^{1/2}+ q^{1/2}\log q) + q^{1/2} \log q.\]
 \end{prop}
In applications, we will choose $a_1$ and $b$ to be such that the main term is  $\geq \al_0 \lfloor N/q \rfloor q^{1/2}$, for some constant $\al_0$. Thus in order for the error term to be, say, at most half the size  of the main term, we will need to have $V \ll_k N^{-1},$ with an implicit constant that is chosen to be appropriately small relative to $k, \al_0$.

\begin{proof} 
By partial summation (see e.g. the standard statement in (\ref{partialsumexp}) below),
 \beq\label{sum_error}
 \sum_{M<n \leq M+N} e(  P(n))
  = e( (y-2\pi b/q)(M+N))\sum_{M<n \leq M+N} e(2\pi (a_1n^k + bn)/q) +E'\eeq
  in which 
 \beq\label{E'}
 |E'| \leq \sup_{u \in [0,N]} \Bigg| \sum_{M < n \leq M+u} e(2\pi (a_1n^k + bn)/q)\Bigg| \cdot VN.
 \eeq
 Now we recognize that for any $u\geq 0$,
 \[ \sum_{M < n \leq M+u} e(2\pi (a_1n^k + bn)/q) \ll_k \lfloor u/q \rfloor q^{1/2}+ q^{1/2}\log q.
 \]
 The first term comes from applying the Weil bound (\ref{Weil}) to as many complete sums as possible, and the second term comes from applying (\ref{completed}) to the one possible remaining incomplete sum.
 In particular, since this upper bound is increasing with $u$ we can apply it in (\ref{E'}) to see that 
 \[|E'| \ll_k NV(\lfloor N/q \rfloor q^{1/2}+ q^{1/2}\log q).\]

The final step is to note that we can also break the main term in (\ref{sum_error}) into $\lfloor N/q \rfloor$ complete sums of length $q$, and at most one incomplete sum that is bounded above by $\ll_k q^{1/2} \log q,$ which contributes an acceptable term to the error. This completes the proof.
\end{proof}

\subsection{Partial summation and integration}
 
 We record here several standard facts that we use to remove slowly-varying weights from sums and integrals.
 First, let $a<b$ be real numbers. Let $\mu$ be an integrable function and let $h$ be a real-valued $C^1$ function. Then integration by parts shows that 
 \beq\label{intbyparts} \int_{a}^b \mu(t) e(h(t))dt = e(h(b))\int_a^b \mu(t)dt +E\eeq
 where $|E| \leq \| \mu \|_{L^1[a,b]}\|h'\|_{L^\infty[a,b]} \cdot (b-a).$
 
 Second, let $\{a_n\}$ be a sequence of complex numbers and let $H$ be a $C^1$ function. Define the partial sum $A(u) = \sum_{M < n \leq u } a_n$. Then as a result of partial summation,
 \beq\label{partialsumint}
 \sum_{M< n \leq M+N}  a_n H(n)  = A(M+N) H(M+N) - \int_{M}^{M+N} A(u) H'(u)du.
 \eeq
 
 Third, let $f$ be a real-valued function and $h$ be a $C^1$ real-valued function. Then as a consequence of (\ref{partialsumint}),
 \beq\label{partialsumexp}
 \sum_{M< n \leq M+N} e(f(n) + h(n)) = e(h(M+N))\sum_{M< n \leq M+N} e(f(n) ) + E\eeq
 where 
 \[|E| \leq \sup_{u \in [0,N]}\Bigg|\sum_{M < n\leq M+u} e(f(n) )\Bigg| \cdot \|h'\|_{L^\infty[M,M+N]} \cdot N.\]

\section{Reduction of the maximal function to   a complete exponential sum}\label{sec_reduction}

\subsection{The initial definition}
Our construction of the counterexample functions $f=f_R$ is motivated by Bourgain's construction in \cite{Bou16} in the special case $P_2(\xi) = |\xi|^2$, as presented and motivated by the third author in \cite{Pie20}. We work initially with unspecified parameters that we will choose optimally at the end of the argument; this reveals both the flexibility and the natural constraints of our method.

We begin by exploiting the simple fact that   modulating a smooth function by an exponential shifts the support of its Fourier transform. Indeed, if 
 $S = (S_1,\dots,S_n) \in \R_{>0}^n$  and we define $S \circ x = (S_1x_1,\dots,S_nx_n)$ and   $S^{-1} = (S_1^{-1},\dots,S_n^{-1})$,
  then
  \[[\Phi(S \circ x) e(M  \cdot x)]\hat{\;}(\xi) = S_1^{-1} \cdots S_n^{-1}\hat{\Phi}(S^{-1} \circ (\xi-M)).\]
Hence if $\hat{\Phi}$ is supported in $B_n(0,1)$ then $[\Phi(S \circ x) e(M  \cdot x)]\hat{\;}(\xi)$ is supported in $B_n(M,\max S_j),$ which in turn lies in an annulus of ``radius'' $M$, if $M$ is appropriately larger than $\max_{1 \leq j \leq n } S_j$.

Once and for all, we fix a Schwartz function $\phi$ on $\R$ that satisfies $\phi \geq 0$,  $\phi (0) = (2\pi)^{-1} \int\hat{\phi}(\xi) d\xi=1,$ and $\supp(\hat{\phi}) \subseteq [-1,1]$. Such a function can be constructed by starting with $\psi \in C_0^\infty (B_1(0,1/4))$ with $(2\pi)^{-1} \int \psi(\xi) d\xi =1$. Then we define $\phi$ by $\hat{\phi} = (2\pi)^{-1} \psi * \overline{\psi(-\cdot)}$, so that $\phi = |\check{\psi}|^2$, in which $\check{\psi}(x) = (2\pi)^{-1} \int \psi(\xi) e^{i x \xi} d\xi.$ 
 Since $\phi$ is fixed once and for all, we will allow implicit constants below to depend on $\phi$ (for example on various bounded norms of $\phi, \phi', \hat{\phi}$), and will often denote this dependence by $\ll_\phi$ without further specification.

Let $R \geq 1$ be given; this is the main parameter we will let go to infinity in Theorem \ref{thm_main}. 
Let \[L=R^{\lambda}, \qquad S_1=R^\sigma\]
for some parameters $0<\lam, \sig < 1$  that we will choose later in terms of $R$.  

Fix $n\geq 2$. Denote $x=(x_1,\dots,x_n)=(x_1,x')$ and define $\Phi_{n-1}(x')=\prod_{j=2}^n \phi(x_j)$. 
Define $f=f_R$ by
\begin{equation}\label{eqn: f}
f(x)=\phi(S_1x_1)e(Rx_1)\Phi_{n-1}(x')\sum_{\substack{m' \in \Z^{n-1} \\ R/L \le m_j < 2R/L}} e(Lm' \cdot x').
\end{equation}
Under the constraints on $L,S_1$ specified above, this function has the property that its Fourier transform is supported in 
\[[R-S_1, R+S_1] \times [R-1, 2R+1]^{n-1} \subseteq B_n(0, \sqrt{n} \cdot 2R + \sqrt{n} S_1) \setminus B_n(0,\sqrt{n} R - \sqrt{n} S_1).\] 
Thus since $S_1=R^\sig$ with $\sig<1$, there exists $R_1 = R_1(n,\sig)$ such that for all $R \geq R_1,$ $\hat{f}$ is supported in the annulus $A_n(R,4 \sqrt{n}).$

\subsection{Computation of $T_t^{(P_k)}f$}
In this section it is convenient to use the notation 
\[ P_k(\xi) = \xi_1^k + \cdots + \xi_n^k = \xi_1^k + \tilde{P}_k(\xi'), \qquad \tilde{P}_k(\xi') := \xi_2^k + \cdots + \xi_n^k.\]
By definition,  
\begin{multline*} 
  T_t^{(P_k)}f(x) =  \frac{1}{2\pi} \int_\R \hat{\phi}(\lambda)e((R+\lambda S_1)x_1+(R+\lambda S_1)^k t) d\lambda  \\
   \times \frac{1}{(2\pi)^{n-1}} \int_{\R^{n-1}} \hat{\Phi}_{n-1}(\xi') \sum_{\substack{m' \in \Z^{n-1} \\ R/L \le m_j < 2R/L}} e((\xi'+Lm') \cdot x'+\tilde{P}_k(\xi'+Lm')t) d\xi'.  
\end{multline*}
 Our goal is to isolate out from this the exponential sum 
    \begin{align}\label{expsumfull}
    \Sbf(2R/L;x',t):=\sum_{\substack{m' \in \Z^{n-1} \\R/L \leq m_j < 2R/L}} e(Lm'\cdot x' + L^k \tilde{P}_k(m') t).
    \end{align}
To do so, we will approximate $T_t^{(P_k)}f(x)$ by an integral that has only linear phase-dependence on $\xi'$ and $\lam$, so that we can apply Fourier inversion and the fact that $\phi$ is nonzero close to the origin. 
The   reduction to $\Sbf(2R/L;x',t)$  generalizes the approach of \cite{Pie20}; the error terms that we accrue in the process are naturally larger than in the quadratic case, since we must remove higher-degree terms from the phase.

Since $\phi(0)=1$ and $\phi$ is smooth, given any small $0<c_0<1/2$ of our choice, there exists a constant $\del_0 = \del_0(c_0,\phi)<1/2$ such that  
 \beq\label{phi_size}
 |\phi(y)| \geq 1-c_0/2, \quad \text{for all $|y| \leq \del_0$}.\eeq

This section proves the following lower bound:
\begin{prop}\label{prop_approx}
Let $0<c_0<1/2$ be a  small constant of our choice, and let $\del_0$ be as in (\ref{phi_size}). Assume $\sig \leq 1/2$. There exist  $0<c_1(k,\del_0),c_2(k,\phi,c_0,\del_0), c_3(k,\phi,c_0)< 1/2$ such that for all sufficiently small constants $c_1<c_1(k,\del_0)$, $c_2<c_2(k,\phi,c_0, \del_0),$ $c_3 < c_3(k,\phi,c_0)$  of our choice, the following holds. 

Let $R \geq R_2(c_1,c_2,n,\phi,\sig)$ be sufficiently large. Let $x \in [-c_1,c_1]^n$ with $x_1\in (-c_1,-c_1/2]$ and $t\in (0,1)$ satisfy the constraints (\ref{eqn: t condition 1 kpurediag}) and (\ref{eqn: t condition 2 kpurediag}) stated below. Then 
\[|T_t^{(P_k)}f(x)| \geq (1-c_0)^n |\Sbf(2R/L;x',t)| + \mathbf{E}_1,\]
in which $|\mathbf{E}_1|$ is bounded as in (\ref{E1_bound}), stated below.
\end{prop}

\subsection{Removal of higher-degree phase in $\lam$, and constraints on $x_1,t$}\label{sec_removal} In this section, we show that the  integral over $\lam$ in $T_t^{(P_k)}f(x)$ has magnitude at least $1-c_0$,  as long as we make appropriate constraints on $x_1$ and $t$. Rewrite the   integral over $\lam$ as
    \begin{eqnarray}
    & &\frac{1}{2\pi} e(Rx_1 + R^k t) \int_\R \hat{\phi}(\lambda) e (\lambda(S_1x_1+k R^{k-1} S_1 t))e \bigg(\sum_{\ell=2}^{k}\binom{k}{\ell} R^{k-\ell} (\lambda S_1)^\ell t\bigg) d\lambda.
    \end{eqnarray}
Since $\hat{\phi}$ is supported in $[-1,1]$, using integration by parts as in (\ref{intbyparts}) to remove the last exponential factor, followed by an application  of Fourier inversion, shows that this expression is equal to
        \begin{eqnarray}
       e(Rx_1+R^k t)e\bigg(\sum_{\ell=2}^{k}\binom{k}{\ell} R^{k-\ell} S_1^{\ell} t \bigg)\phi(S_1(x_1+kR^{k-1}t))+E_1,
    \end{eqnarray}
where the error term has absolute value
    \begin{equation}\label{E1}
    |E_1| \ll_k   \| \hat{\phi} \|_{L^1} \cdot \bigg(\sum_{\ell=2}^{k}\binom{k}{\ell} \ell R^{k-\ell} S_1^{\ell} t\bigg) \cdot 2 \ll_{\phi,k}   tR^k \cdot \sum_{\ell=2}^{k}   (S_1/R)^{\ell}
    \ll_{\phi,k}   tR^k   (S_1/R)^{2}.
    \end{equation}

    Next we place constraints on $t$ so we can bound $|E_1|$ from above and $\phi(S_1(x_1+kR^{k-1}t))$ from below.
We suppose that  $c_0, \del_0$ are as in (\ref{phi_size}). 
Fix another constant $0<c_1< 1/2$; we assume from now on that $x\in [-c_1,c_1]^n$. We then specify two constraints on $t$: first,  we require that for some small $0<c_2<1/2$ of our choice,
\beq\label{eqn: t condition 1 kpurediag}
    t  =-\frac{x_1}{kR^{k-1}}+\tau, \ \text{where } |\tau| \le \frac{c_2\delta_0}{kS_1R^{k-1}}.\eeq
Notice then that by choosing $c_1, c_2$ appropriately small, we can make $t$ as small as any multiple of $1/R^{k-1}$ as we like. In particular, by choosing $c_1$ and $c_2$ sufficiently small relative to $k,\del_0$, we can ensure that for each $j=2,\ldots, n$, 
\beq\label{xj_small}
   | x_j + k(LR')^{k-1}t| \leq \del_0,\eeq
   in which $R' := \lceil 2R/L \rceil -1,$
   and hence $|\phi (x_j + k(LR')^{k-1}t)| \geq 1-c_0/2$. This is a property we will apply momentarily in \S \ref{sec: full integral} below.

    Second, we require that for some small $0<c_3<1/2$ of our choice, 
    \beq\label{eqn: t condition 2 kpurediag}
    |t| \le \frac{c_3}{R^k(S_1/R)^2}.
    \eeq
    In particular, by choosing $c_3$ small enough relative to the implicit constant in the upper bound (\ref{E1}) for $|E_1|$, we  can bound $|E_1|$ by as small a constant as we like, and in particular, as small as $$|E_1| \leq c_0/2.$$

Third, as a consequence of (\ref{eqn: t condition 1 kpurediag}), $|S_1(x_1+kR^{k-1}t)| \leq \del_0$, so that $|\phi(S_1(x_1+kR^{k-1}t))| \geq 1-c_0/2$.
    Thus in total we have confirmed that the integral over $\lam$  in $T_t^{(P_k)}f(x)$ is at least $1-c_0$ in absolute value.
    
    Note that the requirements (\ref{eqn: t condition 1 kpurediag}) and (\ref{eqn: t condition 2 kpurediag}) are compatible as long as we assume that   $S_1 = R^{\sig}$ with 
\beq\label{sig}
\sig \leq 1/2,\eeq
as we do from now on.
    
Finally, note that we can ensure $t\in (0,1)$ by restricting $x_1 \in (-c_1, -c_1/2]$. 
 Then any corresponding $t$ satisfying the above constraints will belong to $(0,1)$ as long as $c_1/(kR^{k-1}) + c_2\del_0/(kS_1R^{k-1}) < 1$ and $c_1/(2kR^{k-1}) - c_2\del_0/(kS_1R^{k-1})>0$.
 This will occur for all sufficiently large $R$, say $R \geq R_2 = R_2(c_1,c_2,n,\phi,\sig).$

At its heart, the efficacy of the counterexamples we construct depends on  $k$ because of the higher order terms we encountered in this step; the constraint $t \ll R^{-(k-1)}$ ultimately affects the size of the prime denominators we use when we construct the set $\Omega^*$ in \S \ref{sec_Omega}.

\subsection{Removal of higher-degree phase components in terms of $\xi'$}\label{sec: remove higher degree}
In this section, we show that the integral over $\xi'$ in the expression for $|T_t^{(P_k)}f(x)|$ can be well-approximated by $(1-c_0)^{n-1}|\Sbf(2R/L;x',t)|$, plus an   error term. We will work one coordinate at a time, and it is convenient to define the one-dimensional exponential sum  
    \beq\label{sum1var}
    S(u; v ,t) := \sum_{R/L \le m<u} e(Lmv+L^k m^k t)
    \eeq
    for any $u \geq R/L.$

    \subsubsection{Expression for the integral over $\xi_j$}
 We begin with the  expression for the integral over $\xi_j$ in $T_t^{(P_k)}f(x)$, namely
\begin{multline}\label{xi_j_int_kpurediag}
     \frac{1}{2\pi} \int_{-1}^1 \hat{\phi}(\xi_j) e(\xi_jx_j + \xi_j^k t)
     \sum_{R/L \le m_j<2R/L} e(Lm_jx_j+L^k m_j^k t)\\
     \times e\bigg( \sum_{\ell=1}^{k-1} \binom{k}{\ell}L^{k-\ell} m_j^{k-\ell}\xi_j^\ell t\bigg) d\xi_j.
\end{multline}
We will show that for $x$ and $t$ as constrained above in (\ref{eqn: t condition 1 kpurediag}) and (\ref{eqn: t condition 2 kpurediag}),  this integral over $\xi_j$ is equal in absolute value to \beq\label{phi_E_2}
|\phi(x_j + k(LR')^{k-1}t)|\cdot|S(2R/L;x_j,t)| + |E_2|,\eeq
in which $R' = \lceil 2R/L \rceil-1$ and
\[ |E_2| \ll_{\phi,k} R^{k-1}|t| \sup_{R/L \leq u <2R/L} |S(u;x_j,t)|.\]
To show this, we first use partial summation to remove the dependency on $m_j$ of any terms involving both $m_j$ and $\xi_j$; this is useful for extracting the sum $S(2R/L;x_j,t)$. Then in a second step we  use integration by parts to remove all higher-order terms in $\xi_j$, to prepare for applying Fourier inversion.

    Let $R'=\lceil 2R/L \rceil-1$. 
    By partial summation as in (\ref{partialsumint}), the sum over $m_j$ is equal to 
   \beq\label{m_main_xi_kpurediag}
         e\bigg( \sum_{\ell=1}^{k-1}  \binom{k}{\ell} (LR')^{k-\ell} \xi_j^\ell t\bigg)S(2R/L;x_j,t) + E_3\eeq
 in which 
 \[ |E_3| \leq \int_{R/L}^{2R/L} |S(u;x_j,t)|\cdot \left|  \sum_{\ell=1}^{k-1}  \binom{k}{\ell}(k-\ell)L^{k-\ell}  u^{k-\ell-1}\xi_j^\ell t \right|du.\]
  Thus in particular for $|\xi_j| \leq 1$,
 \[|E_3| \ll_k (R/L)L^{k-1} (R/L)^{k-2}|t| \sup_{R/L \leq u \leq 2R/L} |S(u;x_j,t)|.\]
 If we then denote by $E_4$ the contribution of this error term to the integral over $\xi_j$, then
  \[ |E_4| \ll_k \|\hat{\phi}\|_{L^1[-1,1]} R^{k-1}|t| \sup_{R/L \leq u \leq 2R/L} |S(u;x_j,t)|.\]

Now we consider the contribution of the main term (\ref{m_main_xi_kpurediag}) to the integral (\ref{xi_j_int_kpurediag}), which is the expression
  \beq\label{xi_j_main_int_kpurediag}
  S(2R/L;x_j,t) \frac{1}{2\pi} \int_{-1}^1  \hat{\phi}(\xi_j) e (\xi_j(x_j +k(LR')^{k-1}t))e\bigg(  \sum_{\ell=2}^{k}  \binom{k}{\ell}(LR')^{k-\ell}  \xi_j^\ell t \bigg) d\xi_j.
  \eeq
In order to apply Fourier inversion we must remove the higher order terms in $\xi_j$, which we can do by applying integration by parts as in (\ref{intbyparts}). This shows that the integral  over $\xi_j$ in the previous line  is equal to 
 \[ e\bigg(  \sum_{\ell=2}^{k} \binom{k}{\ell} (LR')^{k-\ell} t \bigg)  \frac{1}{2\pi} \int_{-1}^1  \hat{\phi}(\xi_j) e(\xi_j(x_j +k(LR')^{k-1}t))d\xi_j + E_5,\]
 in which 
  \[ |E_5| \ll \|\hat{\phi}\|_{L^1[-1,1]}|t|\bigg(  \sum_{\ell=2}^{k}\binom{k}{\ell}\ell(LR')^{k-\ell} \bigg) \ll_{\phi,k} R^{k-2}|t|  .
 \]
 Thus the total contribution of $E_5$ to (\ref{xi_j_main_int_kpurediag}) is $E_6$, say, where 
 \[ |E_6| \ll_{\phi,k} R^{k-2}|t|\cdot |S(2R/L;x_j,t)|.\]
 Finally, we apply Fourier inversion to the main term. This shows that  (\ref{xi_j_main_int_kpurediag}) is equal to 
 \[  S(2R/L;x_j,t) e\bigg(  \sum_{\ell=2}^{k} \binom{k}{\ell} (LR')^{k-\ell} t \bigg) \phi(x_j + k(LR')^{k-1}t) + E_6.\]
We assemble this computation with our upper bound for $|E_4|$, and we conclude that the integral (\ref{xi_j_int_kpurediag}) can be expressed as
  \[   S(2R/L;x_j,t) e\bigg(  \sum_{\ell=2}^{k} \binom{k}{\ell} (LR')^{k-\ell} t \bigg) \phi(x_j + k(LR')^{k-1}t) +E_4+E_6.\]
This gives (\ref{phi_E_2}) and verifies the upper bound 
  \[ |E_2| \leq |E_4| + |E_6| \ll_{\phi,k} R^{k-1}|t| \sup_{R/L \leq u \leq 2R/L}|S(u;x_j,t)|.
  \]
  This proves our claim.

\subsubsection{Expression for the full integral over $\xi'$}\label{sec: full integral}
We now multiply together the expression (\ref{phi_E_2}) we have derived for the absolute value of each integral over $\xi_j$ for $2 \leq j \leq n$. We see that in absolute value, the integral over $\xi'$ in $T_t^{(P_k)}f(x)$ is equal to
\[
\Bigg|\prod_{j=2}^n  \phi(x_j+k(LR')^{k-1}t)\Bigg|\cdot |\Sbf(2R/L;x',t)|+ \Ebf_1,\]
in which the error term   includes all the cross terms accrued when we multiply (\ref{phi_E_2}) for $j=2,\ldots,n$. 
Precisely, we can write
\beq\label{E1_bound}
|\Ebf_1| \ll_{\phi,k} \sum_{\ell=0}^{n-2}   \bigg(  \sup_{2\leq j \leq n}|S(2R/L;x_j,t)|\bigg)^\ell \bigg( R^{k-1}|t|  \sup_{2\leq j \leq n} \sup_{R/L \le u<2R/L}  \left| S(u;x_j,t) \right| \bigg)^{n-1-\ell}.\eeq

Finally, in order to complete the proof of Proposition \ref{prop_approx}, recall the constraint on  $t$ given in (\ref{eqn: t condition 1 kpurediag}), and recall that by choosing $c_1$ and $c_2$ sufficiently small relative to $\del_0,k$, we have (\ref{xj_small}) so that $|\phi(x_j + k(LR')^{k-1}t)| \geq 1-c_0/2 \geq 1-c_0$ for each $j$.
To finish the proof of 
Proposition \ref{prop_approx}, we simply recall that the integral over  $\lam$ is at least $1-c_0$ in absolute value (and at most $\ll_\phi 1$ in absolute value). We apply the first fact when we multiply it by the main term for $\xi'$, and the second fact when we multiply it by the error term for $\xi'$.   This enlarges $\Ebf_1$ by a constant dependent on $\phi$, which we simply include in the implicit constant.
This completes the proof of Proposition \ref{prop_approx}.

In order to be more precise about our upper bound for $\mathbf{E}_1$, we will have to be more specific about the properties of $x$ and $t$. We turn to this in the next section, in which we construct the set $\Omega^*$ to which $x$ belongs. 
Ultimately, we will show in (\ref{bound_E1}) that for $x \in \Omega^*$ there is a choice of $t$ for which
$ |\Ebf_1| \ll_{\phi,k,n} (c_1 + c_2\del_0) (\frac{R}{LQ^{1/2}})^{n-1}$, where $c_1,c_2,\del_0$ can be chosen as small as we like, as in (\ref{eqn: t condition 1 kpurediag}).

We conclude this section by computing the $L^2$ norm, and hence the $H^s$ norm, of $f$.

\subsection{Computation of the $L^2$ norm}\label{sec_norm}
In order to prove  Theorem \ref{thm_main_max}, we must compute the $H^s$ norm of $f$, or in the form of its precursor Theorem \ref{thm_main}, we must compute the $L^2$ norm of $f$.
The norms $\|f\|_{H^s(\R^n)}$ and $\|f\|_{L^2(\R^n)}$ are comparable when $\hat{f}$ is supported in an annulus of radius $R \geq 1$, in which case 
\beq\label{HsL2}
R^s \| f\|_{L^2(\R^n)} \ll_s \|f\|_{H^s(\R^n)} \ll_s R^s \|f\|_{L^2(\R^n)}.
\eeq
Indeed, recall that
the Sobolev space $H^s(\R^n)$ consists of functions $f$ such that $G_{-s} \ast f \in L^2(\R^n)$, where $G_{-s}$ is the Bessel kernel defined by its Fourier transform $\hat{G}_{-s}(\xi) = (1+|\xi|^2)^{s/2}$. By Plancherel's theorem,  
\[
\|f\|_{H^s(\R^n)}^2 = \|G_{-s}*f\|_{L^2(\R^n)}^2 =\frac{1}{(2\pi)^n} \|\hat{G}_{-s}\hat{f}\|_{L^2(\R^n)}^2 = \frac{1}{(2\pi)^n}\int_{\R^n} (1+|\xi|^2)^{s}|\hat{f}(\xi)|^2 d\xi.
\]
In particular if $\hat{f}$ is supported in an annulus $\{(1/C)R\leq |\xi| \leq CR\}$ for a constant $C>1$, then (\ref{HsL2}) holds for all $R \geq 1$.

To compute the $L^2$ norm of $f$, Plancherel's theorem shows that it suffices to compute the $L^2$ norm of $\hat{f}$. By the definition of $f$ in (\ref{eqn: f}),  
\[\hat{f}(\xi_1, \xi') = \sum_{\substack{m'\in \Z^{n-1}\\ R/L \leq m_j < 2R/L
}} g_{m'}(\xi_1, \xi')
\]
with $g_{m'}(\xi_1, \xi') =S_1^{-1} \hat{\phi}(S_1^{-1}(\xi_1 - R))\hat{\Phi}_{n-1}(\xi' - Lm')$. Since $\hat{\phi}(\xi_j)$ is supported in $[-1,1]$, it follows that $g_{m'}$ is supported in $\mathcal{B} + (R,Lm')$ where $\mathcal{B}$ is the box $[-S_1,S_1] \times [-1,1]^{n-1}$. We see that for all sufficiently large $L$ ($L \geq 4$ suffices), the supports of $g_{m'}$ for distinct $m'$ are disjoint. Hence
\[ \|\hat{f}\|_{L^2(\R^n)}^2 = \sum_{\substack{m'\in \Z^{n-1}\\ R/L \leq m_j < 2R/L
}} \|g_{m'}\|_{L^2(\R^n)}^2.
\]
 By Plancherel's theorem again,  \[\|g_{m'}\|_{L^2(\R^n)}^2 = S_1^{-1}\|\hat{\phi}\|_{L^2(\R^n)}^{2n} = S_1^{-1}(2\pi)^{n}\|\phi\|_{L^2(\R^n)}^{2n} .\]
In conclusion,
\beq\label{L2norm}
S_1^{-1/2} \lfloor R/L \rfloor ^{\frac{n-1}{2}} \|\phi\|^n_{L^2(\R)} \leq \|f\|_{L^2(\R^n)} \leq S_1^{-1/2} \lceil R/L \rceil^{\frac{n-1}{2}} \|\phi\|^n_{L^2(\R)}.
\eeq
 In particular, to satisfy the requirements of Theorem \ref{thm_main}, for each value $R$, we can formally define our counterexample function to be $\tilde{f}  = f/\|f\|_{L^2},$ so that it has $L^2$ norm 1. But for simplicity we proceed for now with $f,$ and only apply this normalization in our final arguments in \S \ref{sec_choices}.

\begin{remark}\label{remark_quarter}
For all $n \geq 1,$ our results adapt easily to show that $s \geq 1/4$ is necessary for (\ref{max_op}) to hold. Set $f(x) = \phi(S_1x_1)e(Rx_1)\Phi_{n-1}(x')e(R \cdot x')$, with the understanding that if $n=1$, only the first two factors arise. The method used to prove Proposition \ref{prop_approx} shows that $|T_t^{(P_k)}f(x)| \gg 1$ for a neighborhood of $x$ with measure $\gg 1$. Thus $\| \sup_{0<t<1}|T_t^{(P_k)}f|\|_{L^1(B_n(0,1))}/\|f\|_{L^2(\R^n)} \gg S_1^{1/2}$, with $S_1 =R^\sig$ where $\sig \leq 1/2$. Choosing $\sig=1/2$ leads to the necessary condition $s \geq 1/4.$
\end{remark}

\section{The sets $\Omega$ and $\Omega^*$}\label{sec_Omega}

We have reduced the study of $T_t^{(P_k)}f(x)$ for our function $f$ to the  study of the exponential sum $\Sbf(2R/L;x',t)$ defined in (\ref{expsumfull}). It is now convenient to define new variable names:
\begin{equation}\label{expnot}
s := L^k\tau, \ \ y_1 := -\frac{L^k}{kR^{k-1}}x_1 \ (\mathrm{mod} \ 2\pi), \ \ y_j := Lx_j \ (\mathrm{mod} \ 2\pi).
\end{equation}
 In this notation, we can now write each one-variable sum defined in (\ref{sum1var}) as
\begin{equation}\label{expsum}
S(u;x_j,t)=\sum_{R/L \le m_j<u} e(Lm_jx_j+L^km_j^kt)=\sum_{R/L \le m_j<u} e(m_jy_j+m_j^k(y_1+s)).
\end{equation}

We will define a set $\Omega$ in which the variable $y$ lies, and correspondingly a set $\Omega^*$ in which the variable $x$ lies, such that for each $x \in \Omega^*,$ there is a choice of $t$ such that 
\beq\label{SE}
|\Sbf(2R/L;x',t)| \gg \left\lfloor \frac{R}{Lq} \right\rfloor^{n-1} q^{(n-1)/2}\eeq
 for some prime $q$ in a certain dyadic range $[Q/2,Q]$, where 
 \[ Q = R^\kappa, \qquad 0<\kappa<1\]
   is a parameter that we will later choose optimally to be a small power of $R$.

Our goals for $\Omega$ (and correspondingly $\Omega^*$) have two conflicting priorities. In Theorem \ref{thm_main} we aim to show that 
\[ \|\sup_{0<t<1} |T_t^{(P_k)}f(x)| \|_{L^1(B_n(0,1))} \]
is large. Thus we aim to show that for all $x \in \Omega^*$ we can choose  $t$ to make $|T_t^{(P_k)}f(x)|$ large, and moreover we aim to show that $\Omega^*$ has measure as large as possible, that is, $|\Omega^*| \gg 1$. We will not quite achieve this, but we will show that $|\Omega^*| \gg (\log Q)^{-1}$.

On the other hand, we aim for $\Omega^*$ to have the property that for every $x \in \Omega^*,$ there exists a choice of $t$ such that the previous constraints (\ref{eqn: t condition 1 kpurediag}) and (\ref{eqn: t condition 2 kpurediag}) hold for $t$ and such that all the error terms we have accumulated so far in $\Ebf_1$ (and further error terms we will accumulate in the approximation (\ref{SE})) are sufficiently small. This points to making $\Omega^*$ as small as possible, which clearly conflicts with the first goal. We will work with abstract parameters, and will choose these parameters at the end of the argument in order to optimize the balance between these two goals.

\subsection{Heuristics for a model of the set $\Omega$}
Our model for the set $\Omega$ is as a union of the form:
\[   \Union_{\bstack{Q/2 \leq q \leq Q}{\text{$q$ prime}}} \hspace{1em}
\sideset{}{^*}\Union_{1 \leq a_1 < q} \hspace{1em} \sideset{}{^*}\Union_{a_2,\ldots, a_n } \{ |y_1 - 2\pi a_1/q| <U(q), |y_j - 2\pi a_j/q| <V(q), 2 \leq j \leq n\}.\]
We do not yet specify the widths $U(q)$ and $V(q)$ of the intervals, but so that they do not overlap, we may naturally think of them as functions $q^{-\al} \ll U(q) \ll q^{-\al}$, $q^{-\be} \ll V(q) \ll q^{-\be}$ for some $\al,\be \geq 1$.
The restriction $*$ on the unions indicates for each $q$ we will only choose a certain subset   $a_1,a_2,\ldots,a_n$ of the values $1 \leq a_j \leq q$.  

We wish to restrict the unions to a collection of $a_1,a_2,\ldots,a_n$ chosen so that the complete exponential sum 
 \beq \label{expr_T} T(a_1,a_j;q) = \sum_{1 \leq n \leq q} e(2\pi  (a_1n^k + a_jn)/q) \eeq
is on the order of size  $q^{1/2}$ for each $2 \leq j \leq n.$ 
In fact,  we know that  $\gg q^n$ choices for $a_1,a_2,\ldots, a_n$ lead to this property, by Corollary \ref{cor_T_prod_big}.
Precisely, set $\al_1=1/2$ and $\al_2 = k^{-2}/4$, as in that corollary.
Assume that
\[ Q> 2 \cdot (\al_2/4)^{-1} = 32 k^2,\]
so that all $q \in [Q/2,Q]$ satisfy $q>k$ as well as  $q>(\al_2/4)^{-1}.$
(When we ultimately choose $Q$  to be a small power of $R$, this will hold for all $R \geq R_3(n,k).$) 

By the prime number theorem, for each $X \geq 2$ there are $X/\log X + O(X/(\log X)^2)$ primes $q \leq X$. Thus there exists a universal constant $Q_0$ such that for all  $Q \geq Q_0$ there are at least $(1/4)Q/\log Q$ primes $q \in [Q/2,Q]$. (Again, when we choose $Q$ to be a small power of $R$, this will hold for all $R \geq R_4(n,k)$.)

Now for each prime $q \in [Q/2,Q]$  define the good set $\Gcal^*(q)$   to denote the set of $a_1, a_2, \ldots, a_n$ modulo $q$ for which 
\[ \al_1^{n-1} q^{(n-1)/2} \leq |T(a_1,a_2;q)| \cdots |T(a_1,a_n;q)|   \leq (k-1)^{n-1} q^{(n-1)/2}.
\]
Then by Corollary \ref{cor_T_prod_big},
$|\Gcal^*(q)| \geq (\al_2/2)^n(1-2^{-n}) q^n. $

It remains to decide how large the neighborhoods of the rationals should be, in our definition of $\Omega$. 
On the one hand, $U(q)$ and $V(q)$ must be sufficiently large that the set $\Omega$ has positive measure in $[0,1]^n$, independent of $R$ (or losing at most a logarithmic factor of $R$). 

We could be motivated to choose $V(q)$ according to simultaneous Dirichlet approximation in $n-1$ variables, so that the neighborhoods in the last $n-1$ dimensions fill a positive measure set in $[0,2\pi]^{n-1}.$ 
Simultaneous Dirichlet approximation shows that for every $Q \geq 1$, every point $(y_2,\dots,y_n)$ in $[0,2\pi]^{n-1}$ can be approximated by $(2\pi a_2/q,\dots,2\pi a_n/q)$  with accuracy
\begin{equation}
|y_j-2\pi a_j/q | \le \frac{2\pi}{qQ^{1/(n-1)}},  \ \ 2 \le j \le n.
\end{equation}
This would suggest taking
\beq\label{V_lower}
V(q) \gg q^{-(1+1/(n-1))}.
\eeq
In fact, the complementary condition $V(q) \ll q^{-(1+1/(n-1))}$ will arise naturally in Proposition \ref{prop_box_omega} below, which shows how to compute the measure of a union of boxes from the measures of the individual boxes, if the boxes are appropriately well-distributed.

On the other hand,  $U(q)$ must be sufficiently small that given $y_1$ in an interval of length $2U(q)$ centered at $2\pi a_1/q$, if we set $s = y_1 - 2\pi a_1/q$, where $s=L^k\tau$, then  $\tau$ meets the constraints (\ref{eqn: t condition 1 kpurediag}) and (\ref{eqn: t condition 2 kpurediag}). Thus we are motivated to choose (roughly)
$U(q) \approx L^k \tau \approx L^k /(S_1 R^{k-1}).$  

Similarly, $V(q)$ must be sufficiently small that given $y_j$ in an interval of length $2V(q)$ centered at $2\pi a_j/q$, the error accrued when we replace $y_j$ by $2\pi a_j/q$ in the exponential sum $S(2R/L;x_j,t)$ is sufficiently small. This error term is the term $E$ appearing in Proposition \ref{prop_sum_approx}, applied with $N=R/L$. By the remark following that proposition, we are then motivated to choose $V(q) \ll (R/L)^{-1}$, up to a small constant factor of our choice.
 To make this compatible with the previous restriction (\ref{V_lower}) on $V$, we see that the relation 
 \beq\label{qrl}
Q^{-(1+1/(n-1))} \ll (R/L)^{-1}
\eeq
 must be satisfied.

\subsection{Measure considerations for a union of well-distributed sets}
 
Our next goal is to show that the measure of $\Omega$ (and correspondingly of $\Omega^*$) is sufficiently large. In our   degree $k$ setting, our argument  diverges from the previous works \cite{Bou16} and \cite{Pie20} in the quadratic setting. This is because we   construct $\Omega$ as a union
\[ \Union_q \Union_{(a_1,a_2,\ldots,a_n) \in \Gcal^*(q)} I_{q,\underline{a}} \]
of certain boxes $I_{q,\underline{a}}$,
in which the sets $\Gcal^*(q)$ have sufficiently large cardinality but  are otherwise \emph{inexplicit}. Thus the explicit method developed in \cite{Pie20}   to compute the measure of $\Omega$ does not apply. 

Instead we take an abstract approach. We prove   that if a set $I$ is constructed as a union of sets $I_j$, and if these sets $I_j$ are sufficiently well-distributed, then the measure of $I$ is comparable to the sum of the measures of the $I_j$. After we prove this abstract lemma, we use an arithmetic argument (and the primality of $q$) to prove that in our setting, the boxes $I_{q,\underline{a}}$ corresponding to tuples in $\Gcal^*(q)$ are sufficiently well-distributed. Thus we can compute a lower bound for the measure of $\Omega$ by computing the measures of the individual boxes.

\begin{lemma}\label{lemma_union_gen}
 Suppose we have a finite index set $J$ and a collection of measurable sets $\{ I_j\}_{j \in J}$ in $\R^m$.

 (i)  Suppose the  sets $\{I_j\}_{j \in J}$ have bounded overlap, in the sense that there exists a universal constant $C_0$ such that every point lying in the union $\Union_{j \in J} I_j$ lies in at most $C_0$ of the sets $I_j$.
 Then 
 \[ |\Union_{j \in J} I_j|  \geq  C_0^{-1} \sum_{j \in J} |I_j|.\]
 
 (ii) Suppose the sets $\{I_j\}_{j \in J}$ have comparable sizes, in the sense that for all $j \in J$, $B_0 \leq |I_j| \leq B_1$, 
and that the sets are regularly distributed, in the sense that  
\beq\label{jjbound} \# \{ j , j' \in J : I_j \intersect I_{j'} \neq \emptyset \} \leq C_1 |J|.\eeq
Then 
\[ |\Union_{j \in J} I_j|  \geq  \frac{B_0}{B_1C_1} \sum_{j \in J} |I_j|. \]
\end{lemma}

We will apply case (ii), but without any additional work we include (i) as a simpler model case.
Note that the trivial upper bound in (\ref{jjbound}) is $|J|^2$; (\ref{jjbound}) can be thought of as an assumption of bounded overlap on average.

\begin{proof}
 Define a function $f$ acting on $\R^m$ by 
\[  f(x) = \sum_{j \in J} \onebf_{I_j} (x).\]
By the Cauchy-Schwarz inequality, 
\[ |\Union_{j \in J} I_j| = |\supp(f)|   \geq \frac{\|f\|_{L^1(\R^m)}^2}{\|f\|_{L^2(\R^m)}^2}.\]
On the one hand, 
 \[ \|f\|_{L^1}  = \sum_{j \in J} |I_j|  .\]
 On the other hand,
 \[ \|f\|_{L^2}^2 = \int_{\R^m}  \sum_{j, j'  \in J} \onebf_{I_j}(x) \onebf_{I_{j'}}  (x) dx .
 \]
We now apply either of the hypotheses. If (i) holds, then
\[   \|f\|_{L^2}^2 \leq \sum_{j \in J} |I_j| \cdot \# \{ j' : I_j \intersect I_{j'} \neq \emptyset \} \leq C_0 \sum_{j \in J} | I_j| \leq C_0 \|f\|_{L^1}.\]
Thus 
\[ |\Union_{j \in J} I_j|  \geq \frac{\|f\|_{L^1}^2}{C_0\|f\|_{L^1}} = C_0^{-1}\|f\|_{L^1}.\]
Alternatively, suppose that condition (ii) is met. 
Then  
\begin{align*} \|f\|_{L^2}^2 & \leq   B_1\# \{ j , j' \in J : I_j \intersect I_{j'} \neq \emptyset \} \\
&\leq B_1C_1 |J|  
\leq  B_1 C_1 \sum_{j} \frac{|I_j|}{B_0} = C_1 (B_1/B_0) \|f\|_{L^1}.
\end{align*}
Thus
\[ |\Union_{j \in J} I_j|  \geq \frac{\|f\|_{L^1}^2}{C_1 (B_1/B_0)\|f\|_{L^1}} = \frac{B_0}{B_1C_1}\|f\|_{L^1}.\]
\end{proof}

\subsection{Construction of well-distributed boxes centered at rationals}

Consider a set $\Pcal$ of primes with $\Pcal \subset [Q/2,Q]$.
To each such prime $q \in \Pcal$, we associate  a set $\Gcal^*(q)$ of   tuples $(a_1,a_2,\ldots, a_n)$, with 
$\Gcal^*(q) \subset [1,q]^n$. 
We assume that all the sets $\Gcal^*(q)$ are of comparable size, in the sense that there are uniform constants $D_1, D_2$ such that for all $q, q' \in \Pcal$, 
\beq\label{DDG}
D_1 \leq \frac{|\Gcal^*(q)|}{|\Gcal^*(q')|} \leq D_2.\eeq
This will be true in our application by Corollary \ref{cor_T_prod_big} and the fact that $\Pcal$ lies in a dyadic range.

 To each choice of $q \in \Pcal$ and tuple $\underline{a} \in \Gcal^*(q)$ we associate a box centered at $(2\pi a_1/q,\ldots, 2\pi a_n/q)$, denoted by $I_{q,\underline{a}}$. 
Let us suppose that the box has side-length $h_1(q)$ in the first coordinate and $h_2(q)$ in the coordinates $j=2,\ldots, n$.   Assume that 
\beq\label{hal}
 D_3 q^{-\al} \leq h_1(q)\leq D_4 q^{-\al}, \qquad  D_5 q^{-\be} \leq h_2(q)\leq D_6 q^{-\be},
\eeq
for  constants $0< D_3 < D_4< 1$ and $0<D_5<D_6<1$ of our choice, and for some $1 \leq \al,\be \leq 2.$
 In particular $h_i(x)$ is a decreasing function of $x$.

\begin{prop}\label{prop_box_omega}
In the setting described above, if $|\Gcal^*(q)|\gg q^n$ for all $q \in \Pcal \subset [Q/2,Q]$, and $h_2(x) \ll x^{-1-1/(n-1)},$ then 
\[| \Union_{q \in \Pcal} \Union_{\underline{a} \in \Gcal^*(q)} I_{q,\underline{a}}| \gg  \sum_{q \in \Pcal} \sum_{\underline{a} \in \Gcal^*(q)} |I_{q,\underline{a}}|,\]
in which the implicit constant may depend on $n,\al,\be, D_1,D_2$ but is independent of $Q$.

If we further assume that $|\Pcal| \gg Q/\log Q$, then the lower bound is of the form 
\[ | \Union_{q \in \Pcal} \Union_{\underline{a} \in \Gcal^*(q)} I_{q,\underline{a}}| \gg    Q^{n+1}h_1(Q) h_2(Q)^{n-1} (\log Q)^{-1},\]
in which the implicit constant may depend on $n,\al,\be, D_1,D_2$ but is independent of $Q$.

\end{prop}

 Note that since we assume $h_2(x) \ll x^{-1-1/(n-1)}$, in order for the right-hand side to possibly be $\gg (\log Q)^{-1}$, we would need to take $h_1(x) \gg x^{-1}$. Combined with the hypothesis that $h_1(x)\ll x^{-1}$, this determines that $h_1(x) \approx x^{-1}$. Similarly, comparison of the hypothesis with (\ref{V_lower}) determines in our application that $h_2(x) \approx x^{-1-1/(n-1)}$.

\begin{proof}
We check that this setting obeys hypothesis (ii) of the previous lemma.
The measure of each box $I_{q,\underline{a}}$ is 
\[ h_1(Q) h_2(Q)^{n-1}\leq  |I_{q,\underline{a}}| \leq h_1(Q/2) h_2(Q/2)^{n-1} .  \]
In the notation of the previous lemma,  under the assumption on the functions $h_1,h_2$,
\[
\frac{B_0}{B_1} = \frac{h_1(Q) h_2(Q)^{n-1}}{h_1(Q/2) h_2(Q/2)^{n-1}} \gg  1,
\]
independent of $Q$.

We also need to verify (\ref{jjbound}), for which it suffices to show that 
\beq\label{jjbound_app}
\# \{ (q,\underline{a}), (q', \underline{a}') : I_{q,\underline{a}} \intersect I_{q',\underline{a}'} \neq \emptyset\} \ll |\Pcal| \cdot \min_{q \in \Pcal} |\Gcal^*(q)|.\eeq
The contribution where $(q,\underline{a})= (q', \underline{a}')$ as tuples is at most $|\Pcal| \cdot \max_q |\Gcal^*(q)| \ll |\Pcal| \cdot \min_q |\Gcal^*(q)|$,  under the assumption  (\ref{DDG}) that all the sets $\Gcal^*(q) $ are of comparable size. Thus we consider the case where these tuples are not identical. Supposing  $I_{q,\underline{a}} \intersect I_{q',\underline{a}'} \neq \emptyset$ then it must be the case that simultaneously
\begin{align*}
    |a_1/q - a_1'/q'| &\leq (1/2) h_1(q) + (1/2) h_1(q')\\
     |a_j/q - a_j'/q'| & \leq (1/2) h_2(q) + (1/2) h_2(q'), \qquad 2 \leq j \leq n.
\end{align*}

If $q=q'$ then the upper bound (\ref{hal})  assumed on $h_1(q)$ shows that 
$|a_1 - a_1'| \leq D_4 < 1$,
and $|a_j - a_j'| \leq D_6 <1$ for $2 \leq j \leq n$, so that we would obtain $(q,\underline{a}) = (q',\underline{a}'),$ contrary to our assumption.  

Thus it only remains to consider the case with $q \neq q' \in \Pcal \subset [Q/2, Q]$. Then we learn from the above relations that simultaneously 
\begin{align*}
     |a_1q' - a_1'q| &\leq   Q^2 h_1(Q/2)  \\  
     |a_jq' - a_j'q| & \leq Q^2 h_2(Q/2), \qquad 2 \leq j \leq n.
\end{align*}
Note that under the assumptions on $h_1,h_2$, in each case $Q^2 h_i(Q/2) \gg1$.

If primes $q$ and $q'$ with $\gcd(q,q')=1$ are fixed,  then we claim that the representation of any integer by $a_iq' -a_i'q$ with  $1 \leq a_i \leq q, 1 \leq a_i' \leq q'$ is unique. 
Indeed, suppose that there  is also a representation by $1 \leq b_i \leq q, 1 \leq b_i' \leq q'$. Then $b_iq'-b_i'q = a_iq' -a_i'q$, so that $(b_i - a_i)q' = (b_i'-a_i')q$. Then the fact that $\gcd(q,q')=1$  shows that $q|(b_i-a_i)$ and $q'| (b_i'-a_i')$, which suffices to show that $b_i=a_i$ and $b_i'=a_i'$.

Thus once an integer $m$ with $|m| \leq Q^2 h_1(Q/2) $ is fixed, there is (at most) one choice of a pair $a_1,a_1'$ with $a_1q' - a_1'q=m$. Similarly, for each $j=2,\ldots, n$, once an integer $m$ with $|m| \leq Q^2 h_2(Q/2) $ is fixed, there is (at most) one choice of $a_j,a_j'$ with $a_jq' - a_j'q=m$. Thus once $q \neq q' \in \Pcal$   are fixed, we obtain at most $Q^{2n}h_1(Q/2)h_2(Q/2)^{n-1}$ choices of boxes $I_{q,\underline{a}}, I_{q',\underline{a}'}$   that can intersect. 

In total, we have so far shown that
\[ \# \{ (q,\underline{a}), (q', \underline{a}') : I_{q,\underline{a}} \intersect I_{q',\underline{a}'} \neq \emptyset\} \ll |\Pcal|^2 Q^{2n}h_1(Q/2)h_2(Q/2)^{n-1} +  |\Pcal| \cdot \min_{q \in \Pcal} |\Gcal^*(q)|.\]
In order for this to be sufficiently small to verify (\ref{jjbound_app}), we require that 
\beq\label{h2req}
 |\Pcal| Q^{2n}h_1(Q/2)h_2(Q/2)^{n-1} \ll  \min_{q \in \Pcal} |\Gcal^*(q)|.\eeq
We certainly have $|\Pcal| \ll Q/\log Q$ and $h_1(Q) \ll Q^{-1}$. 
If we assume, as in the hypothesis of the proposition, that  
  $ \min_{q \in \Pcal} |\Gcal^*(q)| \gg Q^n$ and $h_2(Q) \ll Q^{-1-1/(n-1)}$, then (\ref{h2req}) is satisfied.
  This concludes the proof that the hypothesis (\ref{jjbound}) of the lemma is satisfied in our setting.
  
We now can apply the lemma, and hence 
\[   | \Union_{q\in \Pcal} \Union_{\underline{a} \in \Gcal^*(q)} I_{q,\underline{a}}| \gg  \sum_{q \in \Pcal} \sum_{\underline{a} \in \Gcal^*(q)} |I_{q,\underline{a}}|.\]
Finally note that each box has measure $|I_{q,\underline{a}}| = h_1(q)h_2(q)^{n-1}\gg h_1(Q) h_2(Q)^{n-1}$. 
If we additionally assume that $|\Pcal| \gg Q/\log Q$, then the lower bound is of the form 
\[ \gg  Q^{n+1}h_1(Q) h_2(Q)^{n-1}/\log Q.\]
\end{proof}

\subsection{Formal definition of $\Omega$}
We now formally construct the set $\Omega$, and Proposition \ref{prop_box_omega} will allow us to conclude immediately that it has the desired measure.

\begin{prop}\label{prop_omega}
Let $Q > \max\{32k^2,Q_0\}.$
Define for each prime $q \in [Q/2,Q]$ the good set $\Gcal^*(q)$   to denote the set of $a_1, a_2, \ldots, a_n$ modulo $q$ for which
\beq\label{Tprod_lower}
(1/2)^{n-1} q^{(n-1)/2} \leq \prod_{j=2}^n |T(a_1,a_j;q)|  \leq (k-1)^{n-1} q^{(n-1)/2} .
\eeq

Let $0<c_4, c_5< 1/16$ be sufficiently small constants of our choice.
Define
\begin{multline}\label{Omega_dfn}
    \Omega=  \Union_{\bstack{Q/2 \leq q \leq Q}{\text{$q$ prime}}} \hspace{1em}
 \Union_{(a_1,a_2,\ldots,a_n) \in \Gcal^*(q) } \{ |y_1 - 2\pi a_1/q| < c_4 q^{-1},  \\
 |y_j - 2\pi a_j/q| < c_5 q^{-1-1/(n-1)}, 2 \leq j \leq n\}.
\end{multline} 
Then 
\beq\label{Omega_measure}|\Omega| \gg_{n,k,c_4,c_5} (\log Q)^{-1}.\eeq
\end{prop}

\begin{proof}
Apply Proposition \ref{prop_box_omega}  to the boxes $I_{q,\underline{a}}$ with side-lengths 
$h_1(q) = 2c_4 q^{-1}$ and $h_2(q) = 2c_5 q^{-1-1/(n-1)}.$
Note that $q^n \ll_{n,k} |\Gcal^*(q)| \leq   q^n $ for all $q \in [Q/2,Q]$ by Corollary
\ref{cor_T_prod_big}.
Hence we conclude that 
\[
|\Omega| \gg_{c_4,c_5,k} Q^{n+1} Q^{-1} (Q^{-1-1/(n-1)})^{n-1} (\log Q)^{-1} \gg_{n,k,c_4,c_5} (\log Q)^{-1}.
\]
\end{proof}

\subsection{Formal definition of $\Omega^*$}\label{sec_x_y}
We have constructed a set $\Omega$ of  $(y_1,y_2,\ldots,y_n) \in [0,2\pi]^n$. Now we   use the change of variables (\ref{expnot}) to define the corresponding set $\Omega^*$ of $(x_1,x_2,\ldots,x_n)$.
For completeness, we briefly note this correspondence, following the analogous argument given in \cite{Pie20}.

Consider the reduction modulo $2\pi$ map $\iota: \R \rightarrow [0,2\pi]$ defined by $\iota(z) = z \pmod{2\pi}$ and the rescaling map $r: \R \rightarrow \R$ defined by $r(z) = Mz$ for some $M$ sufficiently large so that $Mc_1 > 2\pi$, where $c_1$ is the constant fixed just above (\ref{eqn: t condition 1 kpurediag}). From (\ref{expnot}) we have  $y_1=\iota \circ r (-x_1)$ with $M = \frac{L^k}{k R^{k-1}}$; for each $j=2,\ldots, n$ we have  $y_j = \iota \circ r (x_j)$ with $M = L$. If we assume that $L=R^\lam$ with $\lam> (k-1)/k$ (as we will later verify), then there exists an absolute constant $R_5(k,c_1)$ such that for all  $R \geq R_5$  each of these rescaling factors $M$ is sufficiently large relative to $c_1$.
Finally, let $\pi_j$ be the projection map to the $j$th coordinate.  

Define $\Omega^*\subseteq [-c_1,-c_1/2] \times [-c_1,c_1]^{n-1}$ to be the set such that 
\begin{alignat*}{2}
    \pi_1(\Omega^*) &=  -(\iota \circ r)^{-1}\pi_1(\Omega), & &\qquad M = \frac{L^k}{k R^{k-1}},\\
    \pi_j(\Omega^*) &= (\iota \circ r)^{-1}\pi_j(\Omega), & &\qquad M = L, \qquad  2\leq j \leq n.
\end{alignat*} 

To see that $\Omega^*$ has the desired measure, we may work coordinate by coordinate, since each of $\Omega$ and $\Omega^*$ is  a union of boxes. Let $S_0$ be a set in $[0,2\pi]$. For $S_1\subseteq [-Mc_1, Mc_1]$ with $\iota(S_1) = S_0$, we see that $S_1$ contains at least $2\lfloor Mc_1/2\pi \rfloor$ shifted copies of $S_0$ and so $|S_1|\geq 2\lfloor Mc_1/2\pi \rfloor |S_0|$. Further, for $S_2\subseteq [-c_1,c_1]$ with $r(S_2)=S_1$,    $|S_2| = |S_1|/M$ and so $|S_2| \gg c_1 |S_0|$. Analogously, for $S_1\subseteq [-Mc_1, -Mc_1/2]$ with $\iota(S_1) = S_0$ and $S_2\subseteq [-c_1,-c_1/2]$ with $r(S_2)=S_1$, we have $|S_1|\geq \lfloor Mc_1/4\pi \rfloor |S_0|$ and so $|S_2| = |S_1|/M \gg c_1 |S_0|$. It follows that for $S_0 = \pi_1(\Omega)$, we achieve $|\pi_1(\Omega^*)| \gg c_1 |\pi_1(\Omega)|$ and for $S_0 = \pi_j(\Omega)$ with $j=2,...,n$, we have $|\pi_j(\Omega^*)| \gg c_1 |\pi_j(\Omega)|$. Then $|\Omega^*| \gg_{n,c_1}   |\Omega|$.

In combination with (\ref{Omega_measure}), we may conclude that 
\beq\label{Omega_star_measure}
|\Omega^*| \gg_{n,k,c_1,c_4,c_5} (\log Q)^{-1}.\eeq

\section{Analysis of the arithmetic contribution} \label{sec_arithmetic}
Given any $Q > \max\{32k^2,Q_0\}$, we have now constructed a set $\Omega^* \subset [-c_1,c_1]^n \subset B_n(0,1)$ with measure $|\Omega^*| \gg_{n,k,c_1,c_4,c_5} (\log Q)^{-1}$ and such that for every $x \in \Omega^*$ there exists a corresponding $y \in \Omega$, with $\Omega$ defined in (\ref{Omega_dfn}).
Now we restrict our choice of $Q$ relative to $R,L, S_1$, so that the other desired properties of $\Omega^*$ hold.

\begin{prop}\label{prop_ME2}
 Suppose that $Q> \max\{32k^2,Q_0\}$, and that
    \beq \label{conditions}
    \frac{1}{Q}\leq   \frac{L^k}{S_1R^{k-1}}, \qquad \frac{1}{Q^{1+1/(n-1)}}\ll \bigg(\frac{R}{L}\bigg)^{-1}, \qquad
     \frac{R}{L} \gg Q^{1+\Delta_0}
    \eeq
for some constant $0<\Delta_0 \leq 1/(n-1)$.
There exists $0<c_4(c_2,k,\del_0)<1/16$ such  that if for any $c_4<c_4(c_2,k,\del_0)$ and $c_5<1/16$ we define $\Omega^*$ as above, the following holds. 
 
 For each $x \in \Omega^*$, 
there exists a choice of $t \in (0,1)$ satisfying (\ref{eqn: t condition 1 kpurediag}) and (\ref{eqn: t condition 2 kpurediag}) such that
\beq\label{S_and_E}
|\Sbf(2R/L;x',t)| = \Mbf_1(x',t) + \Ebf_2,
\eeq
in which
\begin{align} 
|\Mbf_1(x',t)| &\geq 2^{-2(n-1)}\left(\frac{R}{LQ^{1/2}}\right)^{n-1}  , \label{bound_M1}\\
|\Ebf_2|&  \ll_{n,k}  (c_5+  Q^{-\Delta_0/2}) \bigg(\frac{R}{LQ^{1/2}}\bigg)^{n-1}.
\label{bound_E2}
\end{align}
Here the implied constant can depend on $n,k$ but is independent of $x',t$.
\end{prop}

At this point we can also complete the upper bound for the error term $\Ebf_1$ from (\ref{E1_bound}).

\begin{lemma}\label{lemma_E1E2}
Assume the conditions of Proposition \ref{prop_ME2}, and for each $x \in \Omega^*$, choose $t$ as in Proposition \ref{prop_ME2}. Then  
\[ |\Ebf_1| \ll_{n,k,\phi }(c_1+c_2\del_0) \left( \frac{R}{LQ^{1/2}} \right)^{n-1}.\]
\end{lemma}

We remark on the motivations for the conditions in (\ref{conditions}).
  The first condition ensures that for each $x \in \Omega^*$ there exists a choice of $t$ such that the one-dimensional exponential sum $S(u;x_j,t)$  has a rational leading coefficient; this allows the use of Proposition \ref{prop_sum_approx}. Equivalently, this is the property that we can choose $s$ so that $y_1 + s = 2\pi a_1/q$ for some $a_1,q$ in the definition of $\Omega$. In the construction of $\Omega$ we specify $|s| < c_4 q^{-1} \leq 2c_4Q^{-1}$; moreover recall from \eqref{eqn: t condition 1 kpurediag} and \eqref{expnot} that we must have $|s|=L^k|\tau| \leq c_2\delta_0L^k/(kS_1R^{k-1}).$ The first condition of \eqref{conditions} imposes that these two restrictions are compatible, and then we simply ensure that we choose $c_4$ sufficiently small that $2c_4< c_2\del_0/k$.

 The second and third conditions in (\ref{conditions}) are imposed so that the upper bound for the term  $\Ebf_2$ in \eqref{S_and_E}   is small enough relative to the main term. We can see from these conditions that $Q$ grows with $R$ since as we assumed from the beginning, $L=o(R)$. The second condition
 can be regarded as imposing that $V(q)$ is small enough; this  
 is consistent with our previous condition (\ref{qrl}). The third condition will provide 
 the term $Q^{-\Delta_0/2}$ in the upper bound (\ref{bound_E2}), which can be made satisfactorily small. 
In particular, there exists some $R_6 = R_6(\Del_0, c_5)$ such that for all $R \geq R_6$, $Q^{-\Del_0/2} \leq c_5.$ 
   
 \begin{remark}\label{remark_dep_k} The first condition in (\ref{conditions}) is   the condition that ultimately forces the dependence on $k$ in the threshold for $s$ proved in Theorems \ref{thm_main_max} and \ref{thm_main}. This will be visible when we optimize the choice of parameters in \S \ref{sec_optimize}.  
 \end{remark}
 
\subsection{Contribution of $|\Sbf(2R/L;x',t)|$ and bound for $\Ebf_2$}
We now prove Proposition \ref{prop_ME2}. Fix $Q > \max \{ 32k^2,Q_0\}.$
Fix $x \in \Omega^*$. 
By definition, this point $x$ corresponds to a point $y \in \Omega$, and for this $y \in \Omega$ there exists a prime $q \in [Q/2,Q]$ and a tuple  $1 \leq a_1,a_2,\ldots, a_n \leq q$  in $\Gcal^*(q)$ such that $q \ndiv a_1$, $|y_1 - 2\pi a_1/q| < c_4 q^{-1}$ and $|y_j - 2\pi a_j/q| < c_5 q^{-1-1/(n-1)}$ for each $2 \leq j \leq n$, and such that (\ref{Tprod_lower}) holds. 
We then define $s$ by $y_1 +s = 2\pi a_1/q$, and this defines $t$ accordingly.
Note that $t$ satisfies (\ref{eqn: t condition 1 kpurediag}) and (\ref{eqn: t condition 2 kpurediag}) as long as $c_4$ is sufficiently small relative to $c_2,\del_0.$

\subsubsection{The one dimensional sums}
For each coordinate $2 \leq j \leq n$ we apply Proposition \ref{prop_sum_approx} to show that for each $u \leq 2R/L,$
    \beq \label{bound_Sj} S(u;x_j,t) = \lfloor (u-R/L)/q \rfloor T(a_1,a_j;q) + E_{2,j}
    \eeq
where $T(a_1,a_j;q)$ is defined in (\ref{expr_T}) and 
    \begin{align*} 
    |E_{2,j}| &\ll_k \frac{R}{L} (c_5q^{-1-1/(n-1)})\bigg(\frac{R}{Lq^{1/2}} + q^{1/2}\log q \bigg) + q^{1/2} \log q.
    \end{align*}

From this we will deduce two results: first, 
\beq\label{lower_upper_ME2}
|S(2R/L;x_j,t)| = \bigg\lfloor \frac{R}{Lq} \bigg\rfloor |T(a_1,a_j;q)| + O_k\bigg((c_5 + Q^{-\Del_0/2})\frac{R}{LQ^{1/2}}\bigg).
\eeq
Second, for all $R/L \leq u \leq 2R/L$,
\beq\label{upper_for_E1}
|S(u;x_j,t)| \ll_k \frac{R}{LQ^{1/2}}.
\eeq

To prove both of these, it is useful to simplify the upper bound for $E_{2,j}.$
By the second condition in (\ref{conditions}), 
$R/L \cdot c_5 q^{-1-1/(n-1)} \ll c_5$ for all $q \in [Q/2,Q].$
By the third condition in (\ref{conditions}),
\beq\label{q_above}
q^{1/2} \log q  \ll \frac{R}{LQ^{1+\Del_0}} q^{1/2}\log q \ll \frac{R}{LQ^{1/2}} Q^{-\Del_0} \log Q \ll \frac{R}{LQ^{1/2}} Q^{-\Del_0/2},\eeq
say. Thus
\beq\label{E2j_first}
     |E_{2,j}|
     \ll_k (c_5 + Q^{-\Del_0/2}) \frac{R}{LQ^{1/2}}.
 \eeq

To prove (\ref{lower_upper_ME2}) we simply apply (\ref{bound_Sj}) with $u=2R/L$, and use our bound for $E_{2,j}$. 
  To prove (\ref{upper_for_E1}) we   apply the Weil bound  to   $T(a_1,a_j;q)$ in the main term, using $q \ndiv a_1$. 
  We conclude that for all $u \leq 2R/L$, 
  \[ |S(u;x_j,t)| \ll_k (R/Lq) q^{1/2} + |E_{2,j}|,\]
  which suffices for (\ref{upper_for_E1}).

 \subsubsection{Assembling the one-dimensional sums}   
We multiply together the expression (\ref{lower_upper_ME2}) over $2\leq j \leq n$, to obtain that 
\[ |\Sbf(2R/L;x',t)| =\bigg\lfloor \frac{R}{Lq}\bigg\rfloor^{n-1}\prod_{j=2}^n |T(a_1,a_j;q)| + \Ebf_2.\]
The first term satisfies the lower bound 
\[ \geq \frac{1}{2^{n-1}}\bigg\lfloor\frac{R}{Lq}\bigg\rfloor^{n-1}  q^{(n-1)/2}
\geq
\frac{1}{2^{2(n-1)}}\left(\frac{R}{LQ^{1/2}}\right)^{n-1} .\]
Here first we applied (\ref{Tprod_lower}), then used the fact that
 $\lfloor\frac{R}{Lq}\rfloor \geq \frac{1}{2}\cdot \frac{R}{Lq},$  which holds as long as $\frac{R}{Lq}\geq 2$. This  we can assure for all $q \in [Q/2,Q]$ by our final choices for $R,L,Q$, as long as $R \geq R_7(n,k)$.  
 This
  suffices for (\ref{bound_M1}).
The error term is of the form
    \[|\Ebf_2| \ll
        \sum_{\ell=0}^{n-2} \bigg(\frac{R}{Lq}\sup_{2\leq j \leq n}|T(a_1,a_j;q)|\bigg)^\ell \bigg( (c_5+ Q^{-\Delta_0/2}) \frac{R}{LQ^{1/2}} \bigg)^{n-1-\ell}.
    \]
For all values of $a_1,a_2,\ldots,a_n$ that were chosen in $\Gcal^*(q)$ in the construction of $\Omega,$  \[ \sup_{2\leq j \leq n}|T(a_1,a_j;q)| \leq (k-1) q^{1/2}
\]
by the Weil bound.  
For all $R \geq R_6(\Del_0,c_5)$ we have $c_5 +Q^{-\Del_0/2} \le 2c_5<1$, so that the dominant term occurs when $\ell=n-2$. Then this satisfies the upper bound in (\ref{bound_E2}).

\subsection{Bound for $\Ebf_1$}
At this point we can also complete our upper bound for the term $\Ebf_1$. Fix $x \in \Omega^*$. We can define the same $s$ (and hence $t$) as in the proof of Proposition \ref{prop_ME2}, satisfying (\ref{eqn: t condition 1 kpurediag}) and (\ref{eqn: t condition 2 kpurediag}).
Now apply the upper bound for $|S(u;x_j,t)|$ derived in (\ref{upper_for_E1}) to the expression for $\Ebf_1$ from (\ref{E1_bound}). 
This shows that  
    \[|\Ebf_1| \ll_{\phi,k} \bigg(\frac{R}{LQ^{1/2}}\bigg)^{n-1}\sum_{\ell=0}^{n-2} (R^{k-1}|t|)^{n-1-\ell}.
    \]
By the conditions (\ref{eqn: t condition 1 kpurediag}) and (\ref{eqn: t condition 2 kpurediag}) on $t$, $R^{k-1}|t|< c_1+c_2\del_0<1$ so that 
the dominant term occurs when  $\ell=n-2$. 
Then we achieve the bound 
    \beq \label{bound_E1} |\Ebf_1|\ll_{\phi,k,n} R^{k-1}|t| \bigg(\frac{R}{LQ^{1/2}} \bigg)^{n-1} \ll_{\phi,k,n} (c_1+c_2\del_0) \bigg(\frac{R}{LQ^{1/2}} \bigg)^{n-1}. 
    \eeq
 Here $c_1, c_2, \del_0$ are constants we can choose as small as we like, as in (\ref{eqn: t condition 1 kpurediag}).

\section{Optimization of parameters and concluding arguments}\label{sec_choices}

The results from Proposition \ref{prop_approx}, Proposition \ref{prop_ME2} and Lemma \ref{lemma_E1E2}  show that for every $x\in \Omega^*$, there exists $t\in (0,1)$   such that
\beq\label{T_final}
|T_t^{(P_k)}f(x)| \geq (1-c_0)^n 2^{-2(n-1)}
    \left(\frac{R}{LQ^{1/2}}\right)^{n-1}  
    - (|\Ebf_1| + |\Ebf_2|).
\eeq
Here the error terms $\Ebf_1,\Ebf_2$ satisfy the upper bounds given in  Lemma \ref{lemma_E1E2} and Proposition \ref{prop_ME2}. This is under the conditions (\ref{eqn: t condition 1 kpurediag}) and (\ref{eqn: t condition 2 kpurediag}) for $t$, condition (\ref{sig}) for $S_1 = R^\sig$, and the conditions in (\ref{conditions}) on $R,L,S_1,Q.$  Recall that we may freely choose the small constants $c_0$ and $\del_0=\del_0(c_0)$, as well as $c_1, c_2,c_3, c_4, c_5$ as small as we like, subject to the dependencies we have recorded in the notations $c_1(k,\del_0), c_2(k,\phi,c_0,\del_0), c_3(k,\phi,c_0),c_4(c_2,k,\del_0).$
The implicit constants in the upper bounds for $\Ebf_1,\Ebf_2$ depend only on $n,k, \phi$. We also recall that  there exists $R_6(\Delta_0,c_5)$ such that for all $R\geq R_6$, $Q^{-\Delta_0/2} \leq c_5.$ 

First we fix $c_0$. Upon choosing $c_1,c_2,c_5,\del_0$   small enough relative to $c_0$ and the implicit constants in Lemma \ref{lemma_E1E2} and Proposition \ref{prop_ME2} (which are dependent only on $n,k,\phi$), choosing $c_3, c_4$ suitably small, and then taking $R \geq R_8(n,k,\phi, \Del_0)$ sufficiently large, we can conclude that $|\Ebf_1| + |\Ebf_2|$ is, say, no more than $1/2$ the size of the main term in (\ref{T_final}). 
We now let $R^*(n,k,\phi,\Del_0, \sig, c_0)$ denote the maximum of $R_1,\ldots, R_8.$ 
Then for all $R \geq R^*$, for all $x \in \Omega^*,$
\[ \sup_{0<t<1} |T_t^{(P_k)}(f)(x)|   \geq \frac{1}{2} (1-c_0)^n 2^{-2(n-1)}
    \left(\frac{R}{LQ^{1/2}}\right)^{n-1} ,
\]
    under the conditions we have assumed so far on $R,L, S_1, Q.$
Combining this with the lower bound (\ref{Omega_star_measure}) for the measure  of $\Omega^*$ and the computation for $\|f\|_{L^2}$ in (\ref{L2norm}), we can conclude that
\[ \frac{\|\sup_{0<t<1} |T_t^{(P_k)}f| \|_{L^1(B_n(0,1))}}{\|f\|_{L^2}}
 \gg_{n,k,\phi}  
  \bigg(\frac{R}{LQ^{1/2}}\bigg)^{n-1}S_1^{1/2}(R/L)^{-(n-1)/2} (\log Q)^{-1}.
 \]

\subsection{Choices for the parameters}\label{sec_optimize}

Now to prove Theorem \ref{thm_main} it suffices to show that for each 
\[s <  \frac{1}{4} + \frac{n-1}{4((k-1)n+1)},\]
we can choose $L, Q,S_1$ in terms of $R$ such that  all previous constraints are met, and
    \beq\label{final_ineq}
    \bigg(\frac{R}{LQ^{1/2}}\bigg)^{n-1}S_1^{1/2}(R/L)^{-(n-1)/2}(\log Q)^{-1} \geq A_sR^{s'}
    \eeq
for some $s'>s$ and some $A_s = A_s(n,k,\phi)$.

Let $0<\sig, \lam,\kappa<1$ denote parameters such that $Q=R^\kappa, L=R^\lambda, S_1 = R^\sig.$ Then \eqref{final_ineq} will hold for all sufficiently large $R$ if
\beq\label{s_relation}
s< \frac{n-1}{2} + \frac{\sig}{2} - \frac{(\kappa + \lam)(n-1)}{2}.\eeq
We also require $\sig \leq 1/2$ as in (\ref{sig}). The first and second constraints from \eqref{conditions} are met if
\beq\label{lam_kap_relation}
\sig + (k-1) \leq k \lam + \kappa, \qquad \lam + \kappa \bigg(\frac{n}{n-1}\bigg) \geq 1.\eeq
(The left-hand relation here is the only effect of $k$ on the choice of parameters.)
Finally we will have to check that certain less restrictive restraints are met, namely that $(k-1)/k< \lam$ and that for some $0<\Del_0 \leq 1/(n-1),$ $\kappa(1+\Del_0) \leq 1-\lam.$

By taking a linear combination of the inequalities in (\ref{lam_kap_relation}) (namely  $1/(k-1)$ times the first one plus $(n-1)$ times the second one) we deduce that 
\beq\label{equality}
\lam + \kappa  \geq \frac{n+\sig/(k-1)}{n+1/(k-1)}.\eeq
The relation (\ref{s_relation}) yields the largest upper bound for $s$ when $\lam+ \kappa$ is the smallest;   thus we will choose $\lam, \kappa$ so that equality holds here. Assuming this for the moment, we learn that 
\[ s < \frac{(n-1) + \sig ((k-2)n+2)}{2((k-1)n+1)}.\]
The upper bound is largest when $\sig$ is largest among allowable values, so we take $\sig =1/2$. 
We solve for values of $\lam,\kappa$ satisfying (\ref{lam_kap_relation}) that attain equality in (\ref{equality}); this yields 
\[ \lambda = 1-\frac{n}{2((k-1)n+1)}, \qquad \kappa = \frac{n-1}{2((k-1)n+1)}.\]
(This means we choose $Q$ such that $Q^{-1-1/(n-1)} \approx (R/L)^{-1},$ or in other words, we can take $\Del_0=1/(n-1).$ Finally, we see that this choice of $\lam$ satisfies $\lam>(k-1)/k$ for all $n \geq 1, k\geq 2.$)
This leads to the final constraint that 
\[ s< \frac{kn}{4((k-1)n+1)}.\]
This completes the proof of Theorem \ref{thm_main}, and hence of Theorem \ref{thm_main_max}.

\subsection*{Acknowledgements}
The authors thank Po Lam Yung and Igor Shparlinski for helpful conversations.
During this work, Pierce has been partially supported by NSF   CAREER grant DMS-1652173, a Sloan Research Fellowship, and the AMS Joan and Joseph Birman Fellowship.

\bibliographystyle{alpha}

\bibliography{NoThBibliography}

\newcommand{\etalchar}[1]{$^{#1}$}
\begin{thebibliography}{HLRn{\etalchar{+}}19}

\bibitem[BAD91]{BenDev91}
M.~Ben-Artzi and A.~Devinatz.
\newblock Local smoothing and convergence properties of {S}chr\"{o}dinger type
  equations.
\newblock {\em J. Funct. Anal.}, 101(2):231--254, 1991.

\bibitem[BBCR11]{BBCR11}
Juan~Antonio Barcel\'{o}, Jonathan Bennett, Anthony Carbery, and Keith~M.
  Rogers.
\newblock On the dimension of divergence sets of dispersive equations.
\newblock {\em Math. Ann.}, 349(3):599--622, 2011.

\bibitem[BDG16]{BDG16}
J.~Bourgain, C.~Demeter, and L.~Guth.
\newblock Proof of the main conjecture in {V}inogradov's mean value theorem for
  degrees higher than three.
\newblock {\em Ann. of Math. (2)}, 184(2):633--682, 2016.

\bibitem[Bou95]{Bou95}
Jean Bourgain.
\newblock Some new estimates on oscillatory integrals.
\newblock In {\em Essays on {F}ourier analysis in honor of {E}lias {M}. {S}tein
  ({P}rinceton, {NJ}, 1991)}, volume~42 of {\em Princeton Math. Ser.}, pages
  83--112. Princeton Univ. Press, Princeton, NJ, 1995.

\bibitem[Bou13]{Bou13}
J.~Bourgain.
\newblock On the {S}chr\"{o}dinger maximal function in higher dimensions.
\newblock {\em Tr. Mat. Inst. Steklova}, 280:53--66, 2013.

\bibitem[Bou16]{Bou16}
J.~Bourgain.
\newblock A note on the {S}chr\"{o}dinger maximal function.
\newblock {\em J. Anal. Math.}, 130:393--396, 2016.

\bibitem[Bou17]{Bou17b}
J.~Bourgain.
\newblock On the {V}inogradov mean value.
\newblock {\em Proc. Steklov Inst. Math.,}, 296:30--40, 2017.

\bibitem[Car80]{Car80}
L.~Carleson.
\newblock Some analytic problems related to statistical mechanics.
\newblock In {\em Euclidean Harmonic Analysis ({P}roc. {S}em., {U}niv.
  {M}aryland, {C}ollege {P}ark, {M}d., 1979)}, volume 779 of {\em Lecture Notes
  in Math.}, pages 5--45. Springer, Berlin, 1980.

\bibitem[Car85]{Car85}
Anthony Carbery.
\newblock Radial {F}ourier multipliers and associated maximal functions.
\newblock In {\em Recent progress in {F}ourier analysis ({E}l {E}scorial,
  1983)}, volume 111 of {\em North-Holland Math. Stud.}, pages 49--56.
  North-Holland, Amsterdam, 1985.

\bibitem[CLS21]{CLS21}
Erin Compaan, Renato Luc\`a, and Gigliola Staffilani.
\newblock Pointwise convergence of the {S}chr\"{o}dinger flow.
\newblock {\em Int. Math. Res. Not. IMRN}, (1):599--650, 2021.

\bibitem[CLV12]{CLV12}
Chu-Hee Cho, Sanghyuk Lee, and Ana Vargas.
\newblock Problems on pointwise convergence of solutions to the
  {S}chr\"{o}dinger equation.
\newblock {\em J. Fourier Anal. Appl.}, 18(5):972--994, 2012.

\bibitem[Cow83]{Cow83}
Michael~G. Cowling.
\newblock Pointwise behavior of solutions to {S}chr\"{o}dinger equations.
\newblock In {\em Harmonic analysis ({C}ortona, 1982)}, volume 992 of {\em
  Lecture Notes in Math.}, pages 83--90. Springer, Berlin, 1983.

\bibitem[CS88]{ConSau88}
P.~Constantin and J.-C. Saut.
\newblock Local smoothing properties of dispersive equations.
\newblock {\em J. Amer. Math. Soc.}, 1(2):413--439, 1988.

\bibitem[CS20]{CheShp20}
Changhao Chen and Igor~E. Shparlinski.
\newblock On large values of {W}eyl sums.
\newblock {\em Adv. Math.}, 370:107216, 48, 2020.

\bibitem[Del74]{Del74}
P.~Deligne.
\newblock La conjecture de {W}eil {I}.
\newblock {\em Inst. Hautes \'{E}tudes Sc. Publ. Math. No.}, 43:273--307, 1974.

\bibitem[DG16]{DemGuo16}
C.~Demeter and S.~Guo.
\newblock Schr\"odinger maximal function estimates via the pseudoconformal
  transformation.
\newblock arXiv:1608.07640, 2016.

\bibitem[DGL17]{DGL17}
X.~Du, L.~Guth, and X.~Li.
\newblock A sharp {S}chr\"odinger maximal estimate in ${\R^2}$.
\newblock {\em Ann. of Math.}, 186:607--640, 2017.

\bibitem[DK82]{DahKen82}
B.~E.~J. Dahlberg and C.~E. Kenig.
\newblock A note on the almost everywhere behavior of solutions to the
  {S}chr\"{o}dinger equation.
\newblock In {\em Harmonic analysis ({M}inneapolis, {M}inn., 1981)}, volume 908
  of {\em Lecture Notes in Math.}, pages 205--209. Springer, Berlin-New York,
  1982.

\bibitem[DKWZ20]{DKWZ20}
X.~Du, J.~Kim, H.~Wang, and R.~Zhang.
\newblock Lower bounds for estimates of the {S}chr\"odinger maximal function.
\newblock {\em Math. Res. Lett.}, 27:687--692, 2020.

\bibitem[DLLZ18]{DGLZ18}
X.~Du, L.Guth, X.~Li, and R.~Zhang.
\newblock Pointwise convergence of {S}chr\"{o}dinger solutions and multilinear
  refined {S}trichartz estimates.
\newblock {\em Forum Math. Sigma}, 6:e14, 18, 2018.

\bibitem[DS20]{DimSee20}
Evangelos Dimou and Andreas Seeger.
\newblock On pointwise convergence of {S}chr\"odinger means.
\newblock {\em Mathematika}, 66(2):356--372, 2020.

\bibitem[DZ19]{DuZha19}
Xiumin Du and Ruixiang Zhang.
\newblock Sharp {$L^2$} estimates of the {S}chr\"{o}dinger maximal function in
  higher dimensions.
\newblock {\em Ann. of Math. (2)}, 189(3):837--861, 2019.

\bibitem[HLRn{\etalchar{+}}19]{HLRRW19}
Jonathan Hickman, Felipe Linares, Oscar~G. Ria\~{n}o, Keith~M. Rogers, and
  James Wright.
\newblock On a higher dimensional version of the {B}enjamin-{O}no equation.
\newblock {\em SIAM J. Math. Anal.}, 51(6):4544--4569, 2019.

\bibitem[IK04]{IwaKow04}
H.~Iwaniec and E.~Kowalski.
\newblock {\em Analytic {N}umber {T}heory}, volume~53.
\newblock Amer. Math. Soc. Colloquium Publications, Providence RI, 2004.

\bibitem[Kat83]{Kat83}
Tosio Kato.
\newblock On the {C}auchy problem for the (generalized) {K}orteweg-de {V}ries
  equation.
\newblock In {\em Studies in applied mathematics}, volume~8 of {\em Adv. Math.
  Suppl. Stud.}, pages 93--128. Academic Press, New York, 1983.

\bibitem[KPV91]{KPV91}
Carlos~E. Kenig, Gustavo Ponce, and Luis Vega.
\newblock Oscillatory integrals and regularity of dispersive equations.
\newblock {\em Indiana Univ. Math. J.}, 40(1):33--69, 1991.

\bibitem[KR83]{KenRui83}
Carlos~E. Kenig and Alberto Ruiz.
\newblock A strong type {$(2,\,2)$} estimate for a maximal operator associated
  to the {S}chr\"{o}dinger equation.
\newblock {\em Trans. Amer. Math. Soc.}, 280(1):239--246, 1983.

\bibitem[KS79]{KniSok79}
L.~A. Kni\v{z}nerman and V.~Z. Sokolinski\u{\i}.
\newblock Some estimates for rational trigonometric sums and sums of {L}egendre
  symbols.
\newblock {\em Uspekhi Mat. Nauk}, 34(3(207)):199--200, 1979.

\bibitem[Lee06]{Lee06}
S.~Lee.
\newblock On pointwise convergence of the solutions to {S}chr\"{o}dinger
  equations in {$\mathbb{R}^2$}.
\newblock {\em Int. Math. Res. Not.}, pages Art. ID 32597, 21, 2006.

\bibitem[LPV21]{LucPon21}
Renato Lucà and Felipe Ponce-Vanegas.
\newblock Convergence over fractals for the {S}chr\"odinger equation,
  arxiv:2101.02495, 2021.

\bibitem[LR17]{LucRog17}
R.~Luc\`a and K.~M. Rogers.
\newblock Coherence on fractals versus pointwise convergence for the
  {S}chr\"{o}dinger equation.
\newblock {\em Comm. Math. Phys.}, 351(1):341--359, 2017.

\bibitem[LR19a]{LucRog19a}
R.~Luc\`a and K.~M. Rogers.
\newblock Average decay of the {F}ourier transform of measures with
  applications.
\newblock {\em J. Eur. Math. Soc. (JEMS)}, 21(2):465--506, 2019.

\bibitem[LR19b]{LucRog19}
R.~Luc\`a and K.~M. Rogers.
\newblock A note on pointwise convergence for the {S}chr\"{o}dinger equation.
\newblock {\em Math. Proc. Cambridge Philos. Soc.}, 166(2):209--218, 2019.

\bibitem[LRS13]{LRS13}
Sanghyuk Lee, Keith~M. Rogers, and Andreas Seeger.
\newblock On space-time estimates for the {S}chr\"{o}dinger operator.
\newblock {\em J. Math. Pures Appl. (9)}, 99(1):62--85, 2013.

\bibitem[Mon94]{Mon94}
H.~L. Montgomery.
\newblock {\em Ten lectures on the interface between analytic number theory and
  harmonic analysis}, volume~84 of {\em CBMS Regional Conference Series in
  Mathematics}.
\newblock AMS, Providence, RI, 1994.

\bibitem[MV08]{MoyVeg08}
A.~Moyua and L.~Vega.
\newblock Bounds for the maximal function associated to periodic solutions of
  one-dimensional dispersive equations.
\newblock {\em Bull. Lond. Math. Soc.}, 40(1):117--128, 2008.

\bibitem[MVV96]{MVV96}
A.~Moyua, A.~Vargas, and L.~Vega.
\newblock Schr\"{o}dinger maximal function and restriction properties of the
  {F}ourier transform.
\newblock {\em Internat. Math. Res. Notices}, (16):793--815, 1996.

\bibitem[MVV99]{MVV99}
A.~Moyua, A.~Vargas, and L.~Vega.
\newblock Restriction theorems and maximal operators related to oscillatory
  integrals in {$\mathbb{R}^3$}.
\newblock {\em Duke Math. J.}, 96(3):547--574, 1999.

\bibitem[Pie20]{Pie20}
L.~B. Pierce.
\newblock On {B}ourgain's counterexample for the {S}chr\"odinger maximal
  function.
\newblock {\em Quart. J. Math.}, 71:1309--1344, 2020.

\bibitem[Rog08]{Rog08}
Keith~M. Rogers.
\newblock A local smoothing estimate for the {S}chr\"{o}dinger equation.
\newblock {\em Adv. Math.}, 219(6):2105--2122, 2008.

\bibitem[RS10]{RogSee10}
Keith~M. Rogers and Andreas Seeger.
\newblock Endpoint maximal and smoothing estimates for {S}chr\"{o}dinger
  equations.
\newblock {\em J. Reine Angew. Math.}, 640:47--66, 2010.

\bibitem[RV07]{RogVil07}
Keith~M. Rogers and Paco Villarroya.
\newblock Global estimates for the {S}chr\"{o}dinger maximal operator.
\newblock {\em Ann. Acad. Sci. Fenn. Math.}, 32(2):425--435, 2007.

\bibitem[RV08]{RogVil08}
Keith~M. Rogers and Paco Villarroya.
\newblock Sharp estimates for maximal operators associated to the wave
  equation.
\newblock {\em Ark. Mat.}, 46(1):143--151, 2008.

\bibitem[RVV06]{RVV06}
Keith~M. Rogers, Ana Vargas, and Luis Vega.
\newblock Pointwise convergence of solutions to the nonelliptic
  {S}chr\"{o}dinger equation.
\newblock {\em Indiana Univ. Math. J.}, 55(6):1893--1906, 2006.

\bibitem[Sj{\"{o}}87]{Sjo87}
Per Sj{\"{o}}lin.
\newblock Regularity of solutions to the {S}chr\"{o}dinger equation.
\newblock {\em Duke Math. J.}, 55(3):699--715, 1987.

\bibitem[Sj{\"{o}}94]{Sjo94}
Per Sj{\"{o}}lin.
\newblock Global maximal estimates for solutions to the {S}chr\"{o}dinger
  equation.
\newblock {\em Studia Math.}, 110(2):105--114, 1994.

\bibitem[Sj{\"{o}}95]{Sjo95}
Per Sj{\"{o}}lin.
\newblock Radial functions and maximal estimates for solutions to the
  {S}chr\"{o}dinger equation.
\newblock {\em J. Austral. Math. Soc. Ser. A}, 59(1):134--142, 1995.

\bibitem[Sj{\"{o}}97]{Sjo97}
Per Sj{\"{o}}lin.
\newblock {$L^p$} maximal estimates for solutions to the {S}chr\"{o}dinger
  equation.
\newblock {\em Math. Scand.}, 81(1):35--68 (1998), 1997.

\bibitem[Sj{\"{o}}98]{Sjo98}
Per Sj{\"{o}}lin.
\newblock A counter-example concerning maximal estimates for solutions to
  equations of {S}chr\"{o}dinger type.
\newblock {\em Indiana Univ. Math. J.}, 47(2):593--599, 1998.

\bibitem[Sj{\"{o}}02]{Sjo02}
Per Sj{\"{o}}lin.
\newblock Homogeneous maximal estimates for solutions to the {S}chr\"{o}dinger
  equation.
\newblock {\em Bull. Inst. Math. Acad. Sinica}, 30(2):133--140, 2002.

\bibitem[Sj{\"{o}}07]{Sjo07}
Per Sj{\"{o}}lin.
\newblock Maximal estimates for solutions to the nonelliptic {S}chr\"{o}dinger
  equation.
\newblock {\em Bull. Lond. Math. Soc.}, 39(3):404--412, 2007.

\bibitem[Sj{\"{o}}09]{Sjo09}
Per Sj{\"{o}}lin.
\newblock Some remarks on {S}obolev regularity.
\newblock {\em Acta Sci. Math. (Szeged)}, 75(1-2):233--239, 2009.

\bibitem[Sj{\"{o}}11]{Sjo11}
Per Sj{\"{o}}lin.
\newblock Radial functions and maximal operators of {S}chr{\"{o}}dinger type.
\newblock {\em Indiana Univ. Math. J.}, 60(1):143--159, 2011.

\bibitem[Sj{\"{o}}19]{Sjo19JFAA}
Per Sj{\"{o}}lin.
\newblock Two theorems on convergence of {S}chr{\"{o}}dinger means.
\newblock {\em J. Fourier Anal. Appl.}, 25(4):1708--1716, 2019.

\bibitem[SS89]{SjoSjo89}
Peter Sj{\"{o}}gren and Per Sj{\"{o}}lin.
\newblock Convergence properties for the time-dependent {S}chr\"{o}dinger
  equation.
\newblock {\em Ann. Acad. Sci. Fenn. Ser. A I Math.}, 14(1):13--25, 1989.

\bibitem[SS10]{SjoSor10}
Per Sj{\"{o}}lin and Fernando Soria.
\newblock A note on {S}chr{\"{o}}dinger maximal operators with a complex
  parameter.
\newblock {\em J. Aust. Math. Soc.}, 88(3):405--412, 2010.

\bibitem[SS14]{SjoSor14}
Per Sj{\"{o}}lin and Fernando Soria.
\newblock Estimates for multiparameter maximal operators of {S}chr{\"{o}}dinger
  type.
\newblock {\em J. Math. Anal. Appl.}, 411(1):129--143, 2014.

\bibitem[SS20]{SjoStr20}
Per Sj{\"{o}}lin and Jan-Olov Str\"{o}mberg.
\newblock Convergence of sequences of {S}chr\"{o}dinger means.
\newblock {\em J. Math. Anal. Appl.}, 483(1):123580, 23, 2020.

\bibitem[SS21]{SjoStr21}
Per Sj{\"{o}}lin and Jan-Olov Str\"{o}mberg.
\newblock Schr{\"o}dinger means in higher dimensions.
\newblock {\em J. Math. Anal. Appl.}, 504(1):125353, 32, 2021.

\bibitem[TV00]{TaoVar00}
T.~Tao and A.~Vargas.
\newblock A bilinear approach to cone multipliers. {I}. {R}estriction
  estimates.
\newblock {\em Geom. Funct. Anal.}, 10(1):185--215, 2000.

\bibitem[Veg88a]{Veg88b}
L.~Vega.
\newblock {\em {El multiplicador de Schr\"odinger. La funcion maximal y los
  operadores de restricci\'on}}.
\newblock Universidad Aut\'onoma de Madrid, 1988.

\bibitem[Veg88b]{Veg88}
Luis Vega.
\newblock Schr\"{o}dinger equations: pointwise convergence to the initial data.
\newblock {\em Proc. Amer. Math. Soc.}, 102(4):874--878, 1988.

\bibitem[Veg92]{Veg92}
Luis Vega.
\newblock Restriction theorems and the {S}chr\"{o}dinger multiplier on the
  torus.
\newblock In {\em Partial differential equations with minimal smoothness and
  applications ({C}hicago, {IL}, 1990)}, volume~42 of {\em IMA Vol. Math.
  Appl.}, pages 199--211. Springer, New York, 1992.

\bibitem[Wan97]{Wan97}
Sichun Wang.
\newblock On the maximal operator associated with the free {S}chr\"{o}dinger
  equation.
\newblock {\em Studia Math.}, 122(2):167--182, 1997.

\bibitem[Wan06]{Wan06}
Sichun Wang.
\newblock A radial estimate for the maximal operator associated with the free
  {S}chr\"{o}dinger equation.
\newblock {\em Studia Math.}, 176(2):95--112, 2006.

\bibitem[Woo16]{Woo16c}
T.~D. Wooley.
\newblock The cubic case of the main conjecture in {V}inogradov's mean value
  theorem.
\newblock {\em Adv. Math.}, 294:532--561, 2016.

\bibitem[Woo19]{Woo19}
Trevor~D. Wooley.
\newblock Nested efficient congruencing and relatives of {V}inogradov's mean
  value theorem.
\newblock {\em Proc. Lond. Math. Soc. (3)}, 118(4):942--1016, 2019.

\bibitem[WZ19]{WanZha19}
Xing Wang and Chunjie Zhang.
\newblock Pointwise convergence of solutions to the {S}chr\"{o}dinger equation
  on manifolds.
\newblock {\em Canad. J. Math.}, 71(4):983--995, 2019.

\end{thebibliography}

\end{document}